\documentclass[11pt,a4paper,reqno]{amsart}

\usepackage{a4wide,color}

\usepackage[english]{babel}
\usepackage[utf8]{inputenc}

\usepackage{bbm,amsmath,amsthm,epsfig,latexsym,marvosym}
\usepackage{amsfonts}
\usepackage{bigints}
\usepackage{amsmath}
\usepackage{amssymb}
\usepackage{subfig}

\allowdisplaybreaks

\usepackage{esint,xfrac}
\usepackage[colorlinks=true,linkcolor=blue,citecolor=red]{hyperref}

\def\R{\mathbb R}
\def\N{\mathbb N}

\def\Cllc{\mathrm{Ad}}
\def\Clc{\mathrm{Ad}}
\def\Ad{\mathrm{Ad}}

\def\cal{\mathcal}
\def\Ao{{\cal A}}

\def\F{{\cal F}}
\def\G{{\cal G}}
\def\H{{\cal H}}
\def\M{{\cal M}}
\def\L{{\cal L}}

\def\e{\varepsilon}

\def\G{\mathcal{G}}

\def\ov{\overline}

{\left\lbrace\begin{array}{@{}l@{}}}%
{\end{array}\right.}

\def\pa{\partial}

\def\wt{{\rightharpoonup^*}\,}
\def\Id{{\rm Id}\,}

\def\d{\, \mathrm{d}}

\def\dist{{\rm dist}}

\def\ca{\mathbbmss{1}}

\def\00{{\bf 0}}

\newcommand{\cc}{\subset\subset}

\newcommand{\ga}{\mathrm{GA}}

\newcommand{\restr}{%
  \,\raisebox{-.127ex}{\reflectbox{\rotatebox[origin=br]{-90}{$\lnot$}}}\,%
}

\DeclareMathOperator*{\spt}{spt}

\newtheorem{theorem}{Theorem}[section]
\newtheorem{corollary}[theorem]{Corollary}
\newtheorem{proposition}[theorem]{Proposition}
\newtheorem{lemma}[theorem]{Lemma}

\theoremstyle{definition}
\newtheorem{remark}[theorem]{Remark}
\newtheorem{definition}[theorem]{Definition}

\numberwithin{equation}{section}
\numberwithin{figure}{section}

\usepackage{color}

\newcommand{\Oin}{\mathcal{A}(\Omega)}

\title[Gamma Convergence of energy-measures]{On the Gamma convergence of functionals defined over pairs of measures and energy-measures}

\author{Marco Caroccia}
\address{DiMaI, Universit\`a di Firenze, and Scuola Normale Superiore di Pisa}

\author{Riccardo Cristoferi}
\address{Heriot-Watt University \\ Department of Mathematical Sciences \\ Edinburgh EH14 4AS \\ United Kingdom}
\email{caroccia.marco@gmail.com,  r.cristoferi@hw.ac.uk}

\makeindex

\usepackage{soul}

\begin{document}

 \begin{abstract}
 A novel general framework for the study of $\Gamma$-convergence of functionals defined over pairs of measures 
 and energy-measures is introduced.
 This theory allows us to identify the $\Gamma$-limit of these kind of functionals by knowing the $
 \Gamma$-limit of the underlying energies.
 In particular, the interaction between the functionals and the underlying energies results, in the case these
 latter converge to a non continuous energy, in an additional effect in the relaxation process. 
 This study was motivated by a question in the context of epitaxial growth evolution with adatoms.
 Interesting cases of application of the general theory are also presented.
 \end{abstract}

 \maketitle


\section{Introduction}

Mathematical models for epitaxial crystal growth usually assume the interfaces to evolve via the so called Einstein-Nernst relation (see, for instance, \cite{FFLM, FFLM2,  FFLM3,  FFLM4}).
For solid-vapour interfaces, it has been observed in \cite{SpeTer} that the usually neglected adatoms (atoms freely moving on the surface of the crystal) play an important role in the description of the evolution of the interface.
For this reason, Fried and Gurtin in \cite{Gurtin} introduced a model that includes the effect of adatoms,
an additional variable whose evolution in time is coupled to the evolution in time of the interface of the crystal.
In the simple case of a crystal growing on a general shape, \emph{i.e.}, not graph constrained, but without considering elastic effects or surface stress, the free energy of the system reads as
\begin{equation}\label{eq:energyf}
\mathcal{G}(E,u):=\int_{\partial^* E} \psi(u) \,\d\mathcal{H}^{d-1}\,.
\end{equation}
Here $E\subset\R^d$ is a set with finite perimeter representing the shape of the crystal, and
$u\in L^1(\partial^*E;[0,+\infty))$ is the density of the adatoms.
The Borel function $\psi:[0,+\infty)\to(0,+\infty)$ is assumed to be non-decreasing and satisfying $\inf\psi>0$.
From the physical point of view, this latter hypothesis is motivated by the fact that, energetically, even an interface without adatoms cannot be created for free.

The interest in considering the model \eqref{eq:energyf} lies in the intriguing and challenging mathematical questions that are connected to the related evolution equations.
In order to perform numerical simulations of the evolution equations obtained \emph{formally} as the gradient flow of \eqref{eq:energyf}, R\"atz and Voigt in \cite{ratz2006diffuse} (see also \cite{Burger}) considered, for $\e>0$, the following phase field model inspired by the Modica-Mortola functional (see \cite{ modica1987gradient, modica1977limite})
\begin{equation}\label{eq:phasefield}
\mathcal{G}_\e(\phi,u):=\int_{\R^d}\left(\, \frac{1}{\e}W(\phi) + \e|\nabla\phi|^2 \,\right) \psi(u)\,\d x\,.
\end{equation}
Here $\phi\in H^1(\R^d)$ is the phase variable, $W:\R\to[0,+\infty)$ is a continuous double well potential vanishing at $0$ and $1$, and $u\in C^0(\R^d;[0,+\infty))$.
The authors worked with the special case $\psi(t):=1+\frac{t^2}{2}$.
Since the Modica-Mortola functional approximates, in the sense of $\Gamma$-convergence, the perimeter functional, it is expected that the phase field model \eqref{eq:phasefield} approximates the sharp interface energy \eqref{eq:energyf}. 
The identification of the correct $\Gamma$-limit in the weakest topology ensuring compactness along sequences of uniformly bounded energy is currently missing in literature.
The main aim of this work consists in such an identification. 
The identification of the $\Gamma$-limit is an important ingredient for the study of the convergence of the solutions of the gradient flow of the phase field energy \eqref{eq:phasefield} to the solution of the gradient flow of the sharp interface energy \eqref{eq:energyf}.
In \cite{ratz2006diffuse}, this convergence was justified by using formal matched asymptotic expansions.
We would like to draw the attention to the fact that energies of the type \eqref{eq:energyf} are also used in model the evolution and equilibria-configuration of \textit{surfactants} in which a chemical additive is considered to be present on the interface of a bubble, driving the evolution \cite{AceBou, AliCicSig, BaiBarMat, FonMorSla}).

As observed in \cite{caroccia2017equilibria}, the energy \eqref{eq:energyf} is not lower semi-continuous with respect to the $L^1\times w^*$ topology, where the density $u$ is seen as the measure $u\mathcal{H}^{d-1}\restr\partial^* E$ (note that this topology allows for very general cracks in the crystal).
Therefore, the functional $\mathcal{G}$ can not be the $\Gamma$-limit of the functionals $\mathcal{G}_\e$.
Note that the $L^1$ convergence for sets only implies the weak* convergence of the distributional derivative of the characteristic functions of the sets.
Therefore, the relaxation of the functional $\mathcal{G}$ does not follow from the results of \cite{ButtFredd}, for which the weak* convergence of the total variation of the distributional derivative of the characteristic functions of the sets would be  required.
In \cite{caroccia2017equilibria}, the authors identified the relaxed functional $\overline{\mathcal{G}}$ of $\mathcal{G}$ with respect to the $L^1\times w^*$ topology.
The question is then the following: given that the Modica-Mortola functional
\[
\mathcal{F}_\e(\phi):=\int_{\R^d}\left(\, \frac{1}{\e}W(\phi) + \e|\nabla\phi|^2 \,\right) \d x
\]
is known to $\Gamma$-converge to the perimeter functional $\mathcal{F}$, is it true that the functionals
$\mathcal{G}_\e$, which can be seen as \emph{adatom-density weighted} versions of the functionals $\mathcal{F}_\e$, $\Gamma$-converge to $\mathcal{G}$, the \emph{adatom-density weighted} versions of the functional $\mathcal{F}$?\\

The above problem was the motivation to undertake the study of the Gamma convergence of such kind of functionals in a more general framework.
The advantage in doing so is in getting a better insight on the technical reasons leading to the answer of the question, other than developing a theory comprehending a variety of other interesting situations.
We now introduce this general framework.
Let $\Omega\subset\R^d$ be a bounded open set with Lipschitz boundary, and denote by $\Ao(\Omega)$ the family of open subsets of $\Omega$.
For $\e>0$, consider the functionals
\[
\F_\e:L^1(\Omega)\times \Ao(\Omega)\to[0,+\infty]\,,
\quad\quad\quad
\F:L^1(\Omega)\times\Ao(\Omega)\to[0,+\infty]
\]
where each $\F_\e$ is lower semi-continuous in the first variable on each open set $A\in \Ao(\Omega)$. For every $\phi\in L^1$, the maps
\[
\F^\phi_\e(\cdot):=  \F_\e(\phi;\cdot) :\Ao(\Omega)\rightarrow [0,+\infty]\,,
\quad\quad\quad
\F^\phi(\cdot):=  \F(\phi;\cdot): \Ao(\Omega)\rightarrow [0,+\infty]\,,
\]
are assumed to be the restriction of Radon measures on $\Ao(\Omega)$.
Suppose that, for every open set $A\in\Ao(\Omega)$ the family $\{\F_\e(\cdot; A)\}_{\e>0}$ is $\Gamma$-converging in the $L^1$ topology to $\F(\cdot;A)$.
Let $\psi:[0,+\infty)\to(0,+\infty)$ be a Borel function with $\inf\psi>0$ and define the functionals
\begin{equation}\label{eq:g}
\mathcal{G}^{\F_\e}(\phi,\mu):=\int_\Omega \psi\left( \frac{\d\mu}{\d\F^\phi_\e} \right) \d \F^\phi_\e\,
\end{equation}
over couples $(\phi,\mu)$, where $\phi\in L^1(\Omega)$ and $\mu$ is a finite non-negative Radon measure on $\Omega$ absolutely continuous with respect to the measure $\F_{\e}^{\phi}$. $\mathcal{G}^{\F_\e}$ is set to be $+\infty$ otherwise. In the same spirit is defined 
	\[
\mathcal{G}^{\F}(\phi,\mu):=\int_\Omega \psi\left( \frac{\d\mu}{\d\F^\phi} \right) \d \F^\phi\,.
\]

The question we want to investigate is the following: is the $\Gamma$-limit of the family $\{\mathcal{G}^{\F_\e}\}_{\e>0}$ related to the relaxation of the functional $\mathcal{G}^\F$ in the $L^1\times w^*$ topology?\\

This  question is reminiscent of a classical problem studied by Buttazzo and Freddi in \cite{ButtFredd} (see also \cite{Butt}) on the $\Gamma$-convergence of functionals defined over pairs of measures.
Given a sequence of non-negative Radon measures $\nu_n$ on $\Omega$, they studied the $\Gamma$-limit of functionals of the form
\begin{equation}\label{eq:h}
\mathcal{H}_n(\mu):=\int_\Omega f\left(x,\frac{\d \mu}{\d \nu_n}\right) \d \nu_n\,,
\end{equation}
defined over vector valued Radon measures $\mu$ on $\Omega$.
Here $f:\Omega\times \R^N\to[0,+\infty)$ is a continuous function, convex in the second variable.
In \cite{ButtFredd} it is proved that, under suitable assumptions on $f$, if $\nu_n\wt\nu$, then $\mathcal{H}_n\stackrel{\Gamma}{\rightarrow}\mathcal{H}$ with respect to the $w^*$ convergence, where
\[
\mathcal{H}(\mu):=\int_\Omega f\left(x,\frac{\d \mu}{\d \nu}\right) \d \nu
	+ \int_\Omega f^\infty\left(x,\frac{\d \mu^{\perp}}{\d |\mu^{\perp}|}\right) \d |\mu^{\perp}|\,,
\]
where $\mu=\frac{\d \mu}{\d \nu}\nu+\mu^{\perp}$ is the Radon-Nicodym decomposition of $\mu$ with respect to $\nu$, and $f^\infty$ is the recession function of $f$.
Our framework includes their result for scalar valued measures, and with $f$ independent of $x\in\Omega$ and subadditive in the second variable.
Indeed, it is possible to reduce the study of the functionals (1.4) to our setting by taking $\F_\e$ and $\F$ constants.

The novelty of this paper is in the treatment of the problem in this general setting,
where the convergence $\nu_n\wt\nu$ is replaced by the $\Gamma$-convergence of the underlying functionals $\F_\e$ to $\F$.
Since for each $\phi\in L^1(\Omega)$ we will ask that there exists $\{\phi_n\}\subset L^1(\Omega)$ with $\phi_n\to\phi$ in $L^1(\Omega)$ such that $\F_n^{\phi_n}\wt \F^\phi$, in a sense the result in \cite{ButtFredd} can be seen as a pointwise convergence in our setting.
When $\F$ is not continuous in $L^1$, the interaction between the underlying functionals $\F_\e$ and the function $\psi$ gives rise, for a class of non continuous functionals $\F$, to an additional relaxation effect for $\mathcal{G}^\F$.
Because of the technical difficulties we have to deal with, the techniques we employ to prove our results, except for the liminf inequality, are independent from the ones present in \cite{ButtFredd}.\\

\subsection{Main results and idea of the proofs.}
In the main result of this paper, Theorem \ref{thm:mainthm1} we are able to prove that the $\Gamma$-limit of the functionals $\mathcal{G}^{\F_\e}$ is $\overline{\mathcal{G}}$: the relaxation of the functional $\mathcal{G}^\F$ in the $L^1\times w^*$ topology.
In particular, an application of Theorem \ref{thm:mainthm1} (see Proposition \ref{prop:resultperimeter}) is used to prove the $\Gamma$-convergence of $\mathcal{G}_\e$ to $\overline{\mathcal{G}}$.

We focus on a particular class of functionals for which continuity fails.
Since we are assuming lower semi-continuity for the functional $\F$, continuity at some $\phi\in L^1(\Omega)$ fails when
\[
\F(\phi)<\lim_{n\to+\infty}\F(\phi_n)
\]
for some $\{\phi_n\}_{n\in\N}\subset L^1(\Omega)$ with $\phi_n\to\phi$ in $L^1(\Omega)$.
The class of functionals we consider are those for which the above loss of upper semi-continuity holds for all $\phi\in L^1(\Omega)$ and locally in a quantitative way.
Namely, we consider the family of functionals (see Definition \ref{def:Admissible}) for which 
for all $\phi\in L^1(\Omega)$ and for all $r\geq 1$
it is possible to find a sequence $\{\phi_n\}_{n\in \N}\subset L^1(\Omega)$ such that
$\phi_n\to\phi$ in $L^1(\Omega)$, and
\begin{equation}\label{eq:f}
\lim_{n\to+\infty} \F(\phi_n;E) \to r \F^\phi(\phi;E)\,,
\end{equation}
for all Borel sets $E\subset\Omega$ with $\F^\phi(\partial E)=0$.
This class of functions contains some interesting cases, such as the perimeter functional and the total variation functional (see Section \ref{sec:per} and \ref{sec:total} respectively).
In the former case, the validity of \eqref{eq:f} was proved in \cite[Theorem 2]{caroccia2017equilibria} by using a wriggling construction:
given a set of finite perimeter $F\subset\Omega$, local oscillations of the boundary of $\phi=\ca_F$, whose intensity is determined by the factor $r$, were used in order to get \eqref{eq:f}.
For this class of functionals (see Theorem \ref{thm:mainthm1}) the $\Gamma$-limit of the family
$\{\mathcal{G}^{\F_\e}\}_{\e>0}$ is 
\[
	\mathcal{G}^\F_{\mathrm{lsc}} (\phi, \mu):=
		\int_{\Omega} \psi^{cs}\left(\frac{\d \mu}{\d \F^{\phi}} \right)\d \F^{\phi}
			+ \Theta^{cs} \mu^{\perp}(\Omega)\,.
\]
Here $\psi^{cs}$ is the convex sub-additive envelope of $\psi$ (see Definition \ref{def:convex}) and 
\[
\Theta^{cs}:=\lim_{t\rightarrow +\infty} \frac{\psi^{cs}(t)}{t}\,.
\]
Note that $\Theta^{cs}<+\infty$, since $\psi^{cs}$ has at most linear growth at infinity (see Lemma \ref{lem:GeneralFactsAboutConvexSubEnvelope}).
Moreover, $\psi^{cs}\leq\psi^c$. Therefore, the quantitative loss of upper semi-continuity of the functional $\F$ results in having a lower energy density for the limiting functional.\\

We report here the ideas behind the proof of Theorem \ref{thm:mainthm1}.
The main technical difficulty of the paper is the fact that the underlying measures $\F^\phi$ we consider come from the energy $\F$.
The assumptions we require do not seem to be too restrictive.

The liminf inequality (Proposition \ref{propo:lowerbound}) follows easily from classical results on lower semi-continuity of functionals defined over pairs of measures originally proved in \cite{ButtFredd}.

The construction of the recovery sequence is done by using several approximations.
In particular, we pass from $\psi$ to $\psi^{cs}$ in two steps: first from $\psi$ to $\psi^c$, and then from $\psi^c$ to $\psi^{cs}$.
This is possible because $(\psi^{c})^{cs}=\psi^{cs}$: the convex sub-additive envelope of the convex envelope corresponds to the convex sub-additive envelope of the function itself (see Lemma \ref{lem:GeneralFactsAboutConvexSubEnvelope}).

Given $\phi\in L^1(\Omega)$ and a non-negative Radon measure $\mu$ on $\Omega$, we consider its decomposition with respect to $\F^\phi$.
We first treat the singular part and we show that it can be energetically approximated by a finite sum of Dirac deltas whose, in turn, can be approximated by regular functions (Proposition \ref{thm:density}). The main technical difficulty here is in having to deal with general Radon measures $\F^\phi$.

We then turn to the absolutely continuous part of the measure $\mu$.
After having showed  that it is possible to assume the density $u$ to be a piecewise constant function (Proposition \ref{propo:fullRCVRYabsoluteCouple}), we prove that it suffices to approximate the energy with density $\psi^c$ (Proposition \ref{propo:fromConvToSubConv}).

Finally, in Proposition \ref{propo:BorToConv}, given a couple $(\phi,u)$ we construct a sequence of pairs $\{(\phi_n,u_n)\}_{n\in\N}$ such that
\[
\limsup_{n\to+\infty}\int_\Omega \psi(u_n)\d\F^{\phi_n}\leq\int_\Omega \psi^c(u)\d\F^\phi\,.
\]
The technical construction is based on a measure theoretical result, Lemma \ref{lem:Crumble}.
This result allow to disintegrate $\Omega$ in sub-domains containing, asymptotically, a certain percentage of
$\F^\phi(\Omega)$, and such that $\F^\phi$ does not charge mass on their boundaries.\\

We also present applications of our general results to some interesting cases: the perimeter functional, a weighted total variation functional and the classical Dirichlet energy (respectively Subsections \ref{sec:per}, \ref{sec:total} and \ref{sec:dir}).
In the former case, the lack of lower semi-continuity was already provided in \cite[Theorem 2]{caroccia2017equilibria}.
Therefore, using the general theory we developed, we can answer the question raised by the application in Continuum Mechanics (see Proposition \ref{prop:resultperimeter}).

In the second case, the family of approximating functionals we consider is the one introduced by Slep\v{c}ev and Garc\'{i}a-Trillos in the context of point clouds (see \cite{trillos2016continuum}), and that are of wide interest for the community (\cite{breslau2, calatroni17, chagialus, cristoferi18, garciatrillos15, garciatrillos15aAAA,  szlambress2, thorpe17bAAA, thorpe17cAAA, thorpe17AAA,  vanGenCarola}).
The main technical result in studying this case is a wriggling result for the weighted total variation functional (see Proposition \ref{propo:wrigTV}), that allows us to use Theorem \ref{thm:mainthm1} to identify the $\Gamma$-limit in Proposition \ref{prop:resulttotal}.\\

This paper is organized as follows.
In Section \ref{sec:mainresult} we state the main hypotheses and results of the paper.
After introducing the main notation in Section \ref{sec:notation}, we devote Section \ref{sec:proofs} to the proofs of Theorem \ref{thm:mainthm1}.
Finally, the above mentioned applications are treated in Section \ref{sec:appl}.


\section{Notation and preliminaries}\label{sec:notation}

In this section we introduce the basic notation and recall the basic facts we need in the paper.

\emph{Convex and convex sub-additive envelope.}
We collect here some properties of the convex sub-additive envelope of a function that we used in the paper.
Since in this paper we always work with nonnegative functions, in the following all the definitions and statements are adapted to this particular case.

\begin{definition}\label{def:convex}
Let $f:[0,+\infty)\to(0,+\infty)$ be a Borel function.
We define $f^c:[0,+\infty)\to(0,+\infty)$, the \emph{convex envelope} of $f$, and $f^{cs}:[0,+\infty)\to[0,+\infty)$, the \emph{convex sub-additive envelope} of $f$, as
\[
f^{c}(t):=\sup\{\varphi(t) \ | \ \varphi\leq f, \ \text{$\varphi$ convex}\}\,,
\]
and
\[
f^{cs}(t):=\sup\{\varphi(t) \ | \ \varphi\leq f, \ \text{$\varphi$ convex and subadditive}\}\,,
\]
respectively. Moreover, we set
\[
\Theta^{cs}:=\lim_{t\rightarrow +\infty} \frac{f^{cs}(t)}{t}\,.
\]
\end{definition}

The first result is the key one that allows us to construct the recovery sequence in two steps.

\begin{lemma}\label{lem:conv}
Let $f:[0,+\infty)\to(0,+\infty)$ be a Borel function. Then $(f^{c})^{cs}=f^{cs}$.
\end{lemma}

\begin{proof}
It is immediate that, if $g\leq f$ is convex and subadditive, that (since it is in particular convex), we have $g\leq f^c$.
Henceforth $f^{cs}\leq f^c$, yielding
\[
f^{cs}\leq (f^{c})^{cs}\,.
\]
On the other hand, if $g\leq f^{c}$ is convex and subadditive function, it holds in particular $g\leq f$.
Hence, from $g\leq f^{cs}$ we get
\[
(f^c)^{cs}\leq f^{cs}
\]
yielding the desired equality.
\end{proof}

\begin{remark}
Let us note that, in general, $(f^c)^s\neq f^{cs}$.
Here, with $f^s$, we denote the subadditive envelope of a function $f$.
Indeed, in general
\[
(f^c)^s>f^{cs}\,.
\]
As an example, let us consider the function $f(t):=\max\{2|t|-1,1\}$.
Since $f$ is convex we have $f^c=f$ and thus $(f^c)^s=f^s$.
It is possible to check that $f^s$ is not convex.
Therefore, it can not coincide with the convex function $f^{cs}$. 
\end{remark}

The following characterization of $f^c$ is well known (see, for instance, \cite[Remark 2.17 (c)]{braides2002gamma}).

\begin{lemma}\label{lem:convenvel}
Let $f:(0,+\infty)\to(0,+\infty)$ be a Borel function. Then
\[
	f^c(t)= \inf\{\lambda f(t_1)+(1-\lambda)f( t_2) \ | \ \lambda\in [0,1], t_1,t_2\in(0,+\infty),\, \ \lambda t_1+(1-\lambda) t_2=t \}\,,
\]
for all $t\in(0,+\infty)$.
\end{lemma}

A useful geometrical characterization of the convex sub-additive envelope of a convex function has been proved in \cite[Proposition A.9 and Lemma A.11]{caroccia2017equilibria}.

\begin{lemma}\label{lem:GeneralFactsAboutConvexSubEnvelope}
Let $f:[0,+\infty)\to(0,+\infty)$ be a convex function.
Then there exist $\{a_i\}_{i\in \N}\subset\R$, $\{b_i\}_{i\in \N}\subset[0,+\infty)$ such that
	\[
	f^{cs}(t)=\sup_{i\in \N}\{a_i t+ b_i\},
	\quad\quad\quad\quad
	\Theta^{cs}=\sup_{i\in \N}\{a_i\} \,.
	\]
Moreover, there exists $t_0\in(0,+\infty]$ such that $f^{cs}=f$ on $[0,t_0)$, while $f^{cs}$ is linear on $[t_0,+\infty)$.
\smallskip
\end{lemma}

Combining the results of Lemma \ref{lem:conv} and Lemma \ref{lem:GeneralFactsAboutConvexSubEnvelope} we get the following.

\begin{lemma}\label{lem:finalConv}
Let $f:[0,+\infty)\to(0,+\infty)$ be a Borel function.
Then there exists $t_0\in(0,+\infty]$ such that
\begin{equation}\label{eqn:characterizationCS}
	f^{cs}(t)=\left\{\begin{array}{ll} 
	 f^c(t) \quad & \ \text{if  } t\in[0,t_0),\\
	 &\\
	 t\frac{f^c(t_0)}{t_0} \ \quad & \ \text{if  } t\in[t_0,+\infty).
	\end{array}\right.
	\end{equation}
In particular, if $t_0<+\infty$, then
	\[
	\Theta^{cs}
	=\frac{f^c(t_0)}{t_0}\,.
	\]
\end{lemma}

\emph{$\Gamma$-convergence.}
We refer to \cite{DM} for a comprehensive treatment of $\Gamma$-convergence.
In this paper we just need the sequential version of it for metric spaces.

\begin{definition}
Let $(Y,d)$ be a metric space and let $F:Y\rightarrow[0,+\infty]$.
We say that a sequence of functionals $\{F_n\}_{n\in\N}$, where $F_n:Y\rightarrow[0,+\infty]$, $\Gamma$-converges to $F$ with respect to the metric $d$, and we write $F_n\stackrel{\Gamma-\d}{\longrightarrow}F$, if
\begin{itemize}
\item[(i)] For every $x \in Y$ and every $\{x_n\}_{n\in\N}\subset Y$ such that
$x_n\stackrel{\mathrm{d}}{\to}x$,
\[
F(x) \leq \liminf_{n \to +\infty} F_n(x_n)\,;
\]
\item[(ii)] For every $x \in Y$ there exists $\{x_n\}_{n\in\N}\subset Y$ such that $x_n\stackrel{\mathrm{d}}{\to}x$ and
\[
\limsup_{n \to +\infty} F_n(x_n)\leq F(x)\,.
\]
\end{itemize}
\end{definition}

In the proof of Theorem \ref{thm:mainthm1} we will make use of the following.

\begin{remark}\label{rem:Gamma}
Let $x\in Y$.
Assume that, for each $\delta>0$, there exists a sequence $\{x_n\}_{n\in\N}\subset Y$ such that $x_n \to x$ and
\[
\limsup_{n \to +\infty} F_n(x_n)\leq F(x)+\delta\,.
\]
Then, by using a diagonal procedure, it is possible to find a sequence $\{y_n\}_{n\in\N}\subset Y$ with $y_n \to x$ and
\[
\limsup_{n \to +\infty} F_n(y_n)\leq F(x)\,.
\]
\end{remark}

\emph{Radon measures.}
We collect here the main properties of Radon measures we will need in the paper.
For a reference see, for instance, \cite[Section 1.4]{AFP}, and \cite[Section 2]{Maggi}.

\begin{definition}
\label{def:weakstar}
We denote by $\mathcal{M}^+(\Omega)$ the space of finite non-negative Radon measures on $\Omega$.
We say that a sequence $\{\mu_n\}_{n\in  \mathbb N} \subset \mathcal{M}^+(\Omega)$ is \emph{weakly* converging} to $\mu\in \mathcal{M}^+(\Omega)$, and we write $\mu_n \wt \mu$, if
\[
\lim_{n\to +\infty} \int_{\Omega} \varphi \mathrm \d\mu_n = \int_{\Omega} \varphi \mathrm \d\mu
\]
for every $\varphi  \in \mathcal C_0(\Omega)$.
\end{definition}

The following characterisation of weak* convergence will be widely used in the paper.

\begin{lemma}
Let $\{\mu_n\}_{n\in  \mathbb N} \subset \mathcal{M}^+(\Omega)$ such that $\sup_{n\in\N}\mu_n(\Omega)<+\infty$.
Then $\mu_n \wt \mu$, for some $\mu\in\mathcal{M}^+(\Omega)$, if and only if
\begin{equation}\label{borelconv}
\lim_{n\to +\infty} \mu_n(E) = \mu(E)\,,
\end{equation}
for all bounded Borel sets $E \subset\subset \Omega$ such that $\mu(\pa E) = 0$.
\end{lemma}

In order to use the metric definition of $\Gamma$-converge, we need a metric on the space $\mathcal{M}^+(\Omega)$ that induces the weak* topology.
This is possible because $C_0(\Omega)$ is separable. For a proof see, for instance, \cite[Proposition 2.6]{DeLellis}.

\begin{lemma}\label{lem:metric}
There exists a distance $\d_{\mathcal{M}}$ on $\mathcal{M}^+(\Omega)$ such that
\[
\mu_k\wt\mu\quad\Leftrightarrow\quad
\lim_{k\rightarrow\infty}d_{\mathcal M}(\mu_k,\mu)=0\quad\text{ and }\quad \sup_{k\in\N}\mu_k(\Omega)<\infty.
\]
\end{lemma}


\section{Main results}\label{sec:mainresult}
In this section we state the main result of the paper, along with a corollary.
In the following $\Omega\subset\R^d$ will always denote a bounded open set.

\begin{definition}\label{def:Admissible}
Denote by $\Oin$ the family of open subsets of $\Omega$. Let $\F: L^1(\Omega) \times  \Oin 
\rightarrow [0,+\infty]$ be a functional, and set
	\[
	X:=\{ \phi\in L^1(\Omega) \ | \ \F(\phi;\Omega)<+\infty\}.
	\]
We say that $\F$ is an \textit{admissible energy} if it satisfies the following conditions:
\begin{itemize}
\item[(Ad1)] For every open set $A\subset\Omega$, the function $\phi\mapsto \F(\phi;A)$ is lower semi continuous on $L^1$;
\item[(Ad2)] For every $\phi\in X$, the map
$\F^\phi:=\F(\phi;\cdot):\Oin \to[0,+\infty]$ is the restriction of a Radon measure on $\Omega$ to $\Oin$;
\item[(Ad3)] For every $\phi\in X$, and every open set $A\in \Oin $ with $\F(\phi;A)=0$, the following holds:
for every $U\in \Oin$ with $U\cc A$, and for every $\e>0$, there exists $\overline{\phi}\in X$ with $\overline{\phi}=\phi$ on $\Omega\setminus U$ such that
\[
\|\overline{\phi}-\phi\|_{L^1}\leq \e\,,\quad\quad\quad
\F^{\overline{\phi}}(\partial U)=0\,,\quad\quad\quad
0<\F^{\overline{\phi}}(A)=\F^{\overline{\phi}}(U)<\e;
\]
\item[(Ad4)] $\F^\phi$ is \emph{purely lower semicontinuous}. Namely for all $\phi\in X$ and for all $f\in L^1(\Omega,\F^{\phi})$, with $f \geq 1$ $\F^\phi$-a.e., there exists a sequence $\{\phi_n\}_{n\in \N}\subset X$ such that
\[
\phi_n\rightarrow \phi\,\, \text{ in } L^1\,,\quad\quad\quad\quad
\F^{\phi_n}\wt f\F^\phi\,.
\]
\end{itemize}
We denote the class of admissible energies by $\Ad$.
\end{definition}

From now on we will consider our functional $\F$ to be defined on $X$.

\begin{remark}\label{rem:locality}
Note that if $\F\in \Ad$, then from (Ad1) it follows that
\[
\F(\phi;A)=\F(\psi;A)
\]
if $\phi=\psi$ in $A$, for $A\in\Oin$.
Hypothesis (Ad3) is a non-degeneracy hypothesis needed to treat null sets for the measure $\F^\phi$.
\end{remark}

\begin{remark}\label{rmk:equivalentDEFllc}
Note that if $\F$ satisfies Definition \ref{def:Admissible}, it is indeed lower semicontinuous and it has the property that any element $\phi\in X$ can be approached in $L^1$ with a sequence $\{\phi_n\}_{n\in \N}$ which locally increases the energy of the prescribed amount $f\geq 1$ (which acts as a Jacobian).
Indeed, the convergence $\F^{\phi_n}\wt f\F^{\phi}$ implies 
	\[
	\lim_{n\rightarrow +\infty} \F^{\phi_n}(E)=\int_{E} f \d \F^{\phi}
	\]
for all Borel sets $E\subset\subset\Omega$ with $\F^\phi(\partial E)=0$.
In particular this also justifies in (Ad4) the name \textit{purely lower semicontinuous} which encodes the fact that around any point $\phi\in X$ a liminf-type inequality for $\F$ is the sharpest bound that can be expected.
\end{remark}

We now introduce the class of approximating energies.

\begin{definition}\label{def:goodApproximating} 
We say that a sequence $\{\F_n\}_{n\in\N}$ of functionals $\F_n: L^1(\Omega) \times \Oin \rightarrow [0,+\infty]$ is a \textit{good approximating} sequence for an energy $\F\in \Ad$ if 
\begin{itemize}
\item[(GA1)] For every open sets $A\in \Oin $, and every $\{\phi_n\}_{n\in\N}\subset L^1(\Omega)$ with $\phi_n\rightarrow \phi$ in $L^1(\Omega)$, we have
\[
\F(\phi;A)\leq \liminf_{n\rightarrow+\infty} \F_n(\phi_n ;A);
\]
\item[(GA2)] For all $n\in\N$, and $\phi\in L^1(\Omega)$ with $\F_n(\phi;\Omega)<+\infty$,
the map $\F_n^\phi(\cdot):=\F_n(\phi;\cdot)$ is the restriction of a Radon measure on $\Omega$ to $\Oin$;
\item[(GA3)] For every $\phi\in X$ there exists a sequence $\{\phi_n\}_{n\in\N}\subset X$ with
$\phi_n\to\phi$ in $L^1(\Omega)$, such that $\F_n^{\phi_n}$ is non atomic for all $n\in\N$, and
\[
\F^{\phi_n}_n\wt \F^{\phi}\,,
\quad\quad\quad\quad
\F^{\phi_n}_n(\Omega)\to\F^\phi(\Omega)\,.
\]
\end{itemize}
The class of good approximating sequences for $\F$ will be denoted by $\ga(\F)$.
\end{definition}

\begin{remark}\label{rem:cmpt}
It is immediate from the definition, that if $\{\F_n\}_{n\in\N}$ is a good approximating sequence for
$\F$, then $\F_n\stackrel{\Gamma}{\rightarrow}\F$ with respect to the $L^1$ topology.
Hypothesis (GA3) is asking for the existence of a recovery sequence satisfying the additional requirement of recovering the energy also locally.

From (GA1) we deduce the following closeness property: if $\{\phi_n\}_{n\in\N}\subset L^1(\Omega)$ with $\phi_n\to\phi$ in $L^1(\Omega)$ is such that
\[
\sup_{n\in \N} \F_n^{\phi_n}(\Omega)<+\infty\,,
\]
then $\phi\in X$.
\end{remark}

\begin{remark}\label{rmk:nonatomicimpliesGA}
Notice that if $\F\in \Ad$ is non atomic, \textit{i.e.} $\F^{\phi}$ is a non atomic Radon measure for all $\phi\in X$, then the constant sequence $\F_n:=\F$ is a good approximating sequence for $\F$.
\end{remark}

We are now in the position to define the main objects of our study.
\begin{definition}
Let $\psi:[0,+\infty)\to(0,+\infty)$ be a Borel function with $\inf\psi>0$.
For $\F:X\times\Ao(\Omega)\to[0,+\infty]$, satisfying property (Ad2) of Definition \ref{def:Admissible}, we define the $\F$-relative energy
$\mathcal{G}^{\F}: X \times \mathcal{M}^+(\Omega)\rightarrow [0,+\infty]$ as
	\begin{equation}
	\mathcal{G}^{\F} (\phi, \mu):= \left\{\begin{array}{ll}
	\displaystyle \int_{\Omega} \psi\left(u\right)\d \F^{\phi} \ &\ \text{if } \mu=u\F^{\phi}\,,\\
	\text{}\\
	+\infty  \ &\ \text{otherwise}.
	\end{array}
	\right.
	\end{equation}
\end{definition}

The main result of this paper concern the behaviour of sequences of $\F_{n}$-relative energies for a good approximating sequence $\{\F_n\}_n\in \mathrm{GA}(\F)$ of an admissible energy $\F$.
The interaction between the underlying functional $\F$ and the functional $\G^\F$ results in a lower limiting energy density, since $\psi^{cs}\leq\psi^c$.

\begin{theorem}[$\Gamma$-convergence for $\F\in \Cllc$]\label{thm:mainthm1}
Let $\F\in \Cllc$, and $\{\mathcal{F}_n\}_{n\in\N}\in\ga(\F)$.
Then $\mathcal{G}^{\F_n}$ $\Gamma$-converges to $\mathcal{G}^\F_{\mathrm{lsc}}$  with respect to the $L^1\times w^*$ topology, where the functional
\[
\mathcal{G}^\F_{\mathrm{lsc}}:X\times \mathcal{M}^+(\Omega)\rightarrow [0,+\infty]
\]
is defined as
\begin{equation}\label{eqn:sharpenergy1}
	\mathcal{G}^\F_{\mathrm{lsc}} (\phi, \mu):=
		\int_{\Omega} \psi^{cs}\left(\frac{\d \mu}{\d \F^{\phi}} \right)\d \F^{\phi}
			+ \Theta^{cs} \mu^{\perp}(\Omega).
\end{equation}
Here
$\mu=\frac{\d \mu}{\d \F^{\phi}}\F^{\phi}+\mu^{\perp}$ is the Radon-Nicodym decomposition of $\mu$ with respect to $\F^\phi$.
\end{theorem}

\begin{remark}
Note that $\Theta^{cs}<+\infty$, since $\psi^{cs}$ is at most linear at infinity (see Lemma \ref{lem:finalConv}).
\end{remark}

In particular, combining the above theorem with Remark \ref{rmk:nonatomicimpliesGA}, allow to identify, for certain energies in $\Clc$, the relaxation of the $\F$-relative energy $\G^{\F}$ in the $L^1\times w^*$ topology.

\begin{corollary}\label{cor:relcllc}
Let $\F\in \Cllc$ be non atomic. Then the relaxation of $\mathcal{G}^\F$ with respect to the $L^1\times w^*$ topology is $\mathcal{G}^\F_{\mathrm{lsc}}$.
\end{corollary}

\begin{remark}
From the properties of $\psi^c$ and $\psi^{cs}$ and using Remark \ref{rem:cmpt} it is possible to deduce the following compactness property. Assume that a good approximating sequence $\{\F_n\}_{n\in\N}$ enjoys this compactness property:
if $\{\phi_n\}_{n\in\N}\subset L^1(\Omega)$ is such that
\[
\sup_{n\in\N} \F_n(\phi_n)<\infty,
\]
then $\phi_n\to\phi$ in $L^1(\Omega)$ for some $\phi\in X$.
Then, for every $\{\phi_n\}_{n\in\N}\subset L^1(\Omega)$ and $\{\mu_n\}_{n\in\N}\subset \mathcal{M}^+(\Omega)$ such that
\[
\sup_{n\in\N} \mathcal{G}_{lsc}^{\F_n}(\phi_n,\mu_n)<+\infty
\]
it is possible to extract a subsequence (not relabeled) in such a way that $\phi_n\to\phi$ in $L^1$ and $\mu_n\wt\mu$, for some $\phi\in X$ and $\mu\in \mathcal{M}^+(\Omega)$.
\end{remark}


\section{Proof of main theorems}\label{sec:proofs}

Hereafter, $\psi:[0,+\infty)\to(0,+\infty)$ will be a Borel function with $\inf\psi>0$.
We will denote by $\d$ a metric on the space $L^1(\Omega)\times\mathcal{M}^+(\Omega)$ which induces the $L^1\times w^*$ topology (see Lemma \ref{lem:metric}).

\subsection{Liminf inequality}
The proof of the liminf inequality for the $\Gamma$-convergence result of Theorem \ref{thm:mainthm1} follows from the argument in Proposition \cite[Lemma 2.34]{AFP}.
For the reader's convenience, we report it here.

\begin{proposition}\label{propo:lowerbound}
Let $(\phi,\mu)\in L^1(\Omega)\times\mathcal{M}^+(\Omega)$, and $\{(\phi_n,\mu_n)\}_{n\in\N}\subset L^1(\Omega)\times\mathcal{M}^+(\Omega)$ be such that
$(\phi_n,\mu_n)\rightarrow (\phi,\mu)$.
Fix $\F\in \Ad$ and consider a good approximating sequence $\{\F_n\}_{n\in\N}$ for $\F$.
Then
	\begin{align}\label{eqn:liminfcllc}
	\liminf_{n\rightarrow +\infty} \mathcal{G}^{\F_n}(\phi_n,\mu_n)&\geq \int_{\Omega} \psi^{cs}\left(\frac{\d \mu}{\d \F^{\phi}}\right) \d \F^{\phi}+\Theta^{cs} \mu^{\perp}(\Omega) 
\end{align}
\end{proposition}

\begin{proof}
Assume, without loss of generality that
	\[
	\sup_{n\in \N} \{\mathcal{G}^{\F_n} (\phi_n,\mu_n)\}<+\infty\,.
	\]
Therefore $\mu_n=g_n \F_n^{\phi_n}$ for all $n\in\N$.
Since $\inf\psi>0$, using Remark \ref{rem:cmpt} we have that $\phi\in X$.
Note that $\psi\geq \psi^{cs}$, and that, by Lemma \ref{lem:GeneralFactsAboutConvexSubEnvelope} there exists two sequence $\{a_i\}_{i\in \N}\subset \R$, $\{b_i\}_{i\in \N}\subset [0,+\infty)$, such that
\begin{align}\label{eq:char2}
	\psi^{cs}(x)&=\sup_{i\in \N}\{a_i x +b_i \ | \ b_i>0\}, \  &\Theta^{cs}=\sup_{i\in \N}\{a_i\}. 
	\end{align}
Let $\{A_j\}_{j=1}^M$ be a family of pair-wise disjoint of open subset of $\Omega$ and, for each $j\in\N$, let $v_j\in C^{\infty}_c(A_j)$, with $v_j\in [0,1]$.
We have that
	\begin{align*}
		\int_{A_j} \psi(g_n)\d\F_n^{\phi_n}&\geq \int_{A_j} \psi^{cs}\left(g_n\right)\d\F_n^{\phi_n}\geq  \int_{A_j} \left(v_j a_j g_n +b_j\right)\d\F_n^{\phi_n}\\
		&=  \int_{A_j}v_j a_j g_n\d\F_n^{\phi_n}+ b_j\F_n^{\phi_n}(A_j).
	\end{align*}
By summing over $j=1,\ldots,M$, taking the limit as $n\rightarrow +\infty$, exploiting the lower semicontinuity of $\F$, and using the fact that $b_j>0$ together with $\mu_n\wt\mu$, we get
	\begin{align}\label{tt}
		\liminf_{n\rightarrow +\infty} \mathcal{G}^{\F_n} (\phi_n,\mu_n)
			&\geq   \sum_{j=1}^M\int_{A_j}v_j a_j \d \mu+  \sum_{j=1}^M b_j\int_{A_j} v_j \d \F^{\phi}
				\nonumber \\
		&=  \sum_{j=1}^M \int_{A_j}v_j \left(a_j \frac{\d \mu}{\d\F^{\phi}}(x)
			+b_j\right)\d \F^{\phi}+ \sum_{j=1}^M \int_{A_j}v_j a_j \d \mu^{\perp}\,.
		\end{align}
Taking the supremum in \eqref{tt} among all finite families $\{A_j\}_{j\in J}$ of pair-wise disjoint subsets of $\Omega$, and $v_j\in C_c^{\infty}(A_j)$ with $v_j\in[0,1]$, we get
	\begin{align*}
	\liminf_{n\rightarrow +\infty} \mathcal{G}^{\F_n} (\phi_n,\mu_n)\geq	
		\sup\left\{\, \sum_{j\in J} \int_{A_j} \psi^+_j(x) \d \lambda \,\,|\,\, \{A_j\}_{j\in J},\, A_j\subset\Omega \text{ pair-wise disjoint } \,\right\}\,.
	\end{align*}
Here $\psi^+_j(x):=\max\{\psi_j(x),0\}$,
	\begin{equation*}
	 \psi_j(x):=\left\{
	\begin{array}{ll}
	a_j \frac{\d \mu}{\d \F^{\phi}}(x)+b_j \ & \text{on } \Omega\setminus N\, \\
	a_j \ & \text{on }N\,,
	\end{array}\right.
	\end{equation*}
being $N$ a set where $\mu^{\perp}$ is concentrated and $\lambda=\F^{\phi}+\mu^{\perp}$.
Using \cite[Lemma 2.35]{AFP} we can infer
\[
\sup\left\{\, \sum_{j\in J} \int_{A_j} \psi^+_j(x) \d \lambda \,\,|\,\,
	\{A_j\}_{j\in J},\, A_j\subset\Omega \text{ pair-wise disjoint } \,\right\}
	=\int_{\Omega}\,\, \sup_{j\in \N} \{\psi_j(x)^+\} \d \lambda.
\]
Hence, using \eqref{eq:char2}, we get
	\[
	\int_{\Omega} \sup_{j\in \N} \{\psi_j(x)^+\} \d \lambda
		=\int_{\Omega} \psi^{cs}\left(\frac{\d \mu}{\d \F^{\phi}}\right)\d \F^{\phi}
			+\Theta^{cs} \mu^{\perp}(\Omega).
	\]
This concludes the proof.
\end{proof}


\subsection{Limsup inequality}

The proof of the limsup inequality uses several approximations that we prove separately.
We start by proving a density result that will allow us to construct the recovery sequence only for
\textit{absolutely continuous couples}, namely for pairs $(\phi, h\F^\phi)$ with $h\in L^1(\Omega;\F^\phi)$.

\begin{proposition}\label{thm:density}
Let $\F\in \Ad$, $\zeta:[0,+\infty)\to(0,+\infty)$ a convex function such that
\begin{equation}\label{eq:flinearatinfinity}
\Theta_{\zeta}:=\lim_{t\to+\infty}\frac{\zeta(t)}{t}<+\infty\,.
\end{equation}
Then for any $(\phi,\mu)\in X\times \mathcal{M}^+(\Omega)$, there exist $\{\varphi_n\}_{n\in\N}\subset X$, and $\{h_n \}_{n\in \N}$ with $h_n\in L^1(\Omega;\F^{\varphi_n})$ such that 
	\[
	\lim_{n\to+\infty} \d\left( (\varphi_n,h_n\F^{\varphi_n}),(\phi,\mu)\right)=0\,,
	\]
and
	\[
	\lim_{n\to+\infty}\int_{\Omega} \zeta(h_n)\d \F^{\varphi_n}=
		\int_{\Omega} \zeta\left(\frac{\d \mu}{\d \F^{\phi}}\right)\d \F^{\phi}
		+\Theta_{\zeta} \mu^{\perp}(\Omega)\,.
	\]
\end{proposition}

\begin{figure}[h!]
\centering
\subfloat[][\small{\emph{We are given $(\phi, \mu)\in X\times \mathcal{M}^{+}(\Omega)$ where $\mu=g\F^{\phi}+\mu^{\perp}$}}]
{\includegraphics[width=.52\columnwidth]{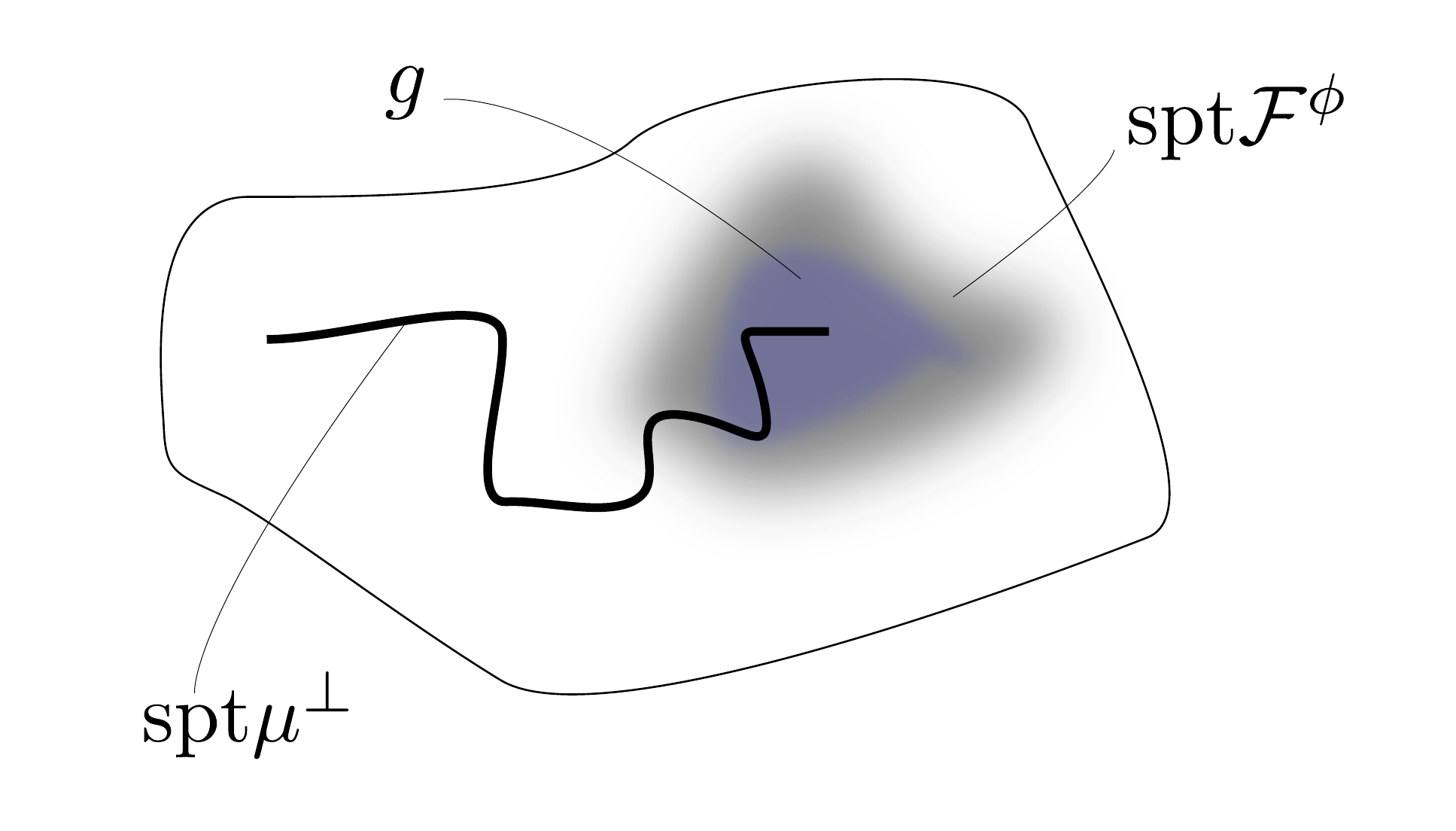}\label{approx0}}
\subfloat[][\small{\emph{We divide $\Omega$ in small cubes and we select only the cubes $Q_j^k$ such that $Q_j^k\cap \spt \mu^{\perp}\neq \emptyset$. We then select $x_j^k\in Q_j^k\cap \spt \mu^{\perp}$ which will provide an approximation of $\mu^{\perp}$ with Dirac deltas. We identify the points of type "a" (in red) as those points $x_j^k\in \spt \F^{\phi}$ and the points of type "b" (in green) as thos points $x_j^k\notin \spt \F^{\phi}$.  }}]
{\includegraphics[width=.52\columnwidth]{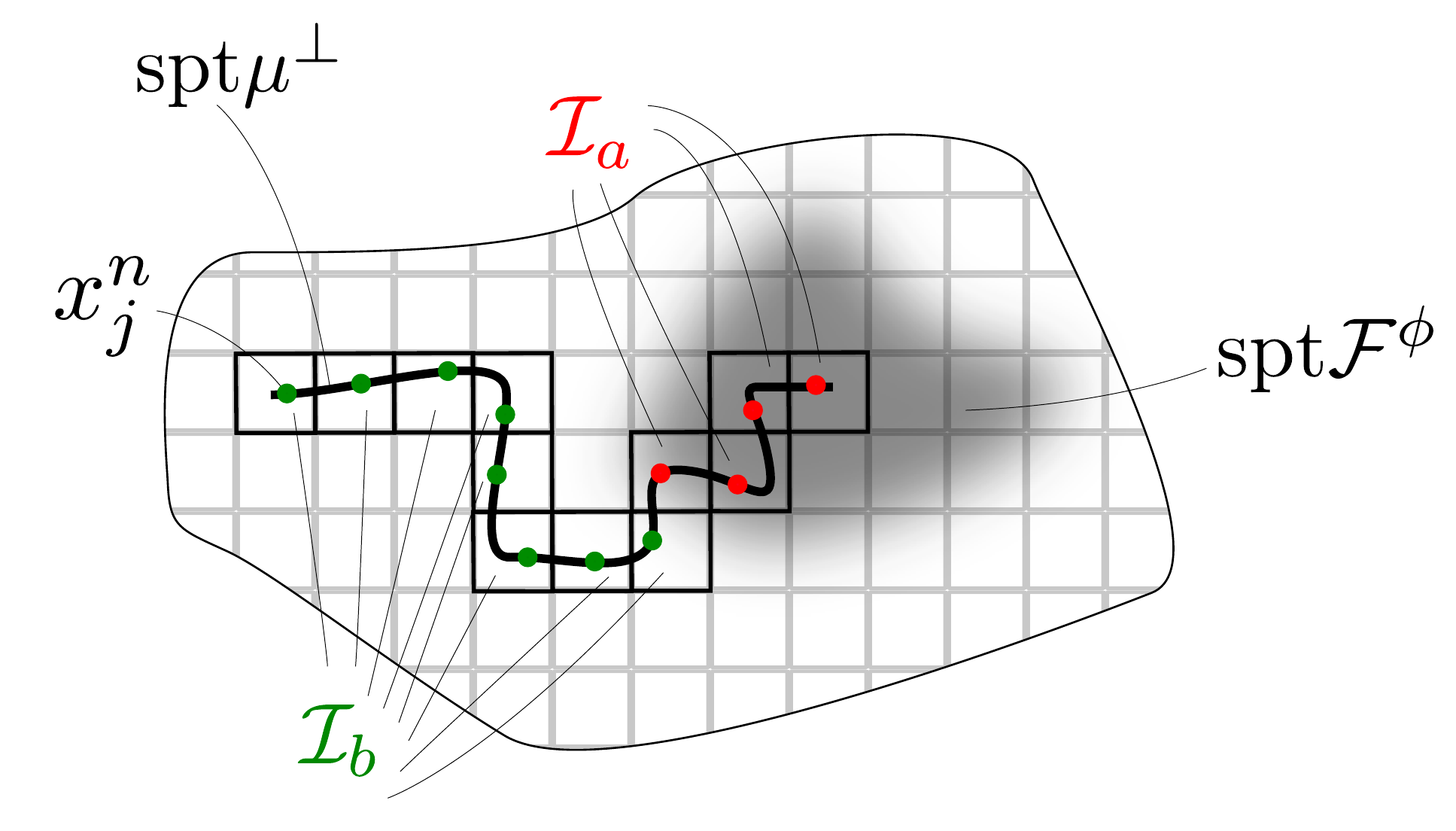}\label{approx1}} \\
\vspace{1cm}
\subfloat[][\small{\emph{We find suitably small balls $B_{r_k}(x_j^k)\cc Q_j^k$ around each point $x_j^k$ where we suitably modify the function $\phi$.
}}]
{\includegraphics[width=.52\columnwidth]{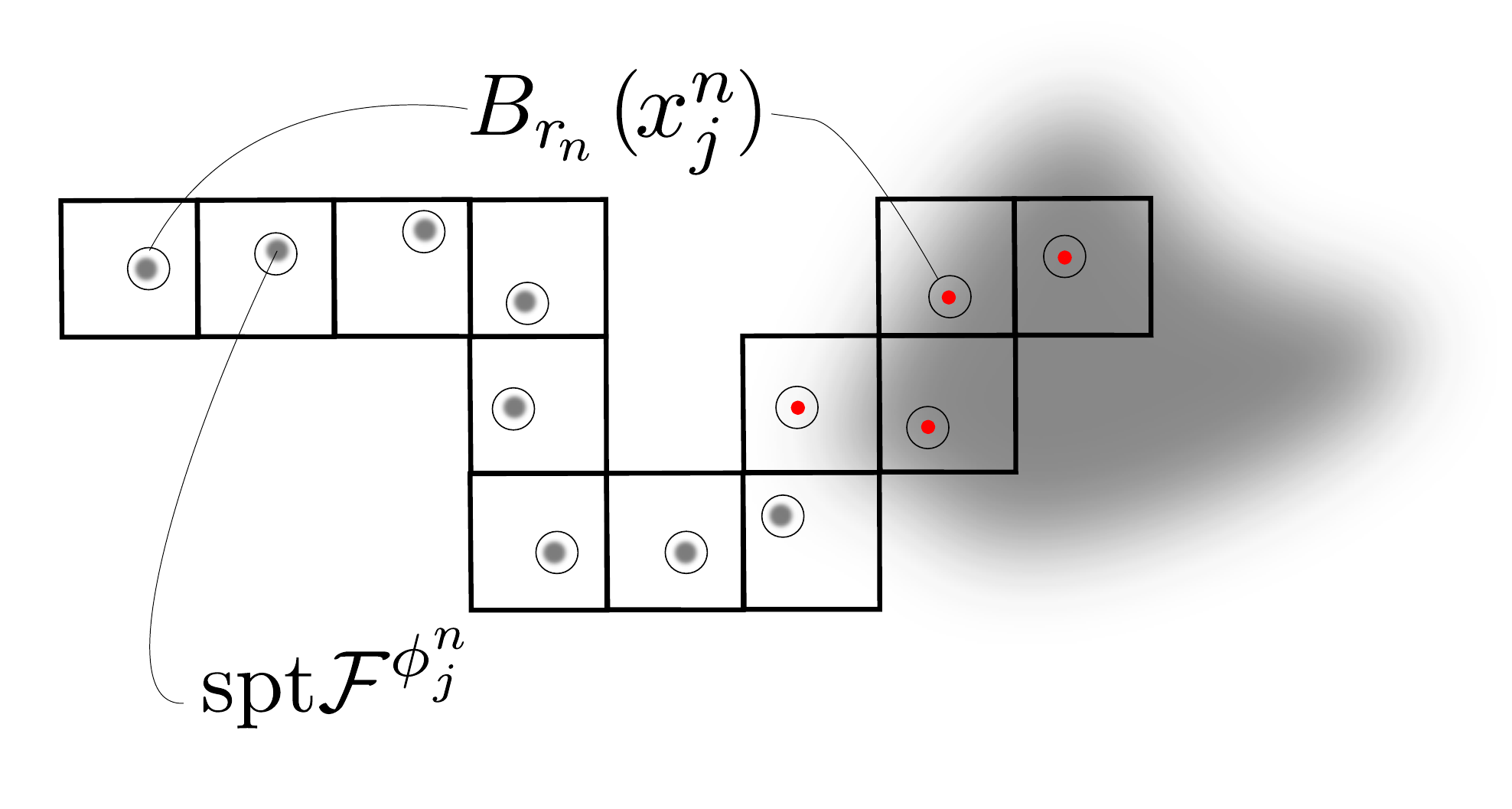}\label{approx2}}
\subfloat[][\small{\emph{We accordingly define functions $h_j^n$ as $\frac{\mu^{\perp}(Q_j^k)}{\F^{\phi}(B_{r_k}(x_j^k))}$  for points of type "a" and as $\frac{\mu^{\perp}(Q_j^k)}{\F^{\phi_j^k}(B_{r_k}(x_j^k))}$ for point of type "b".
}}]
{\includegraphics[width=.52\columnwidth]{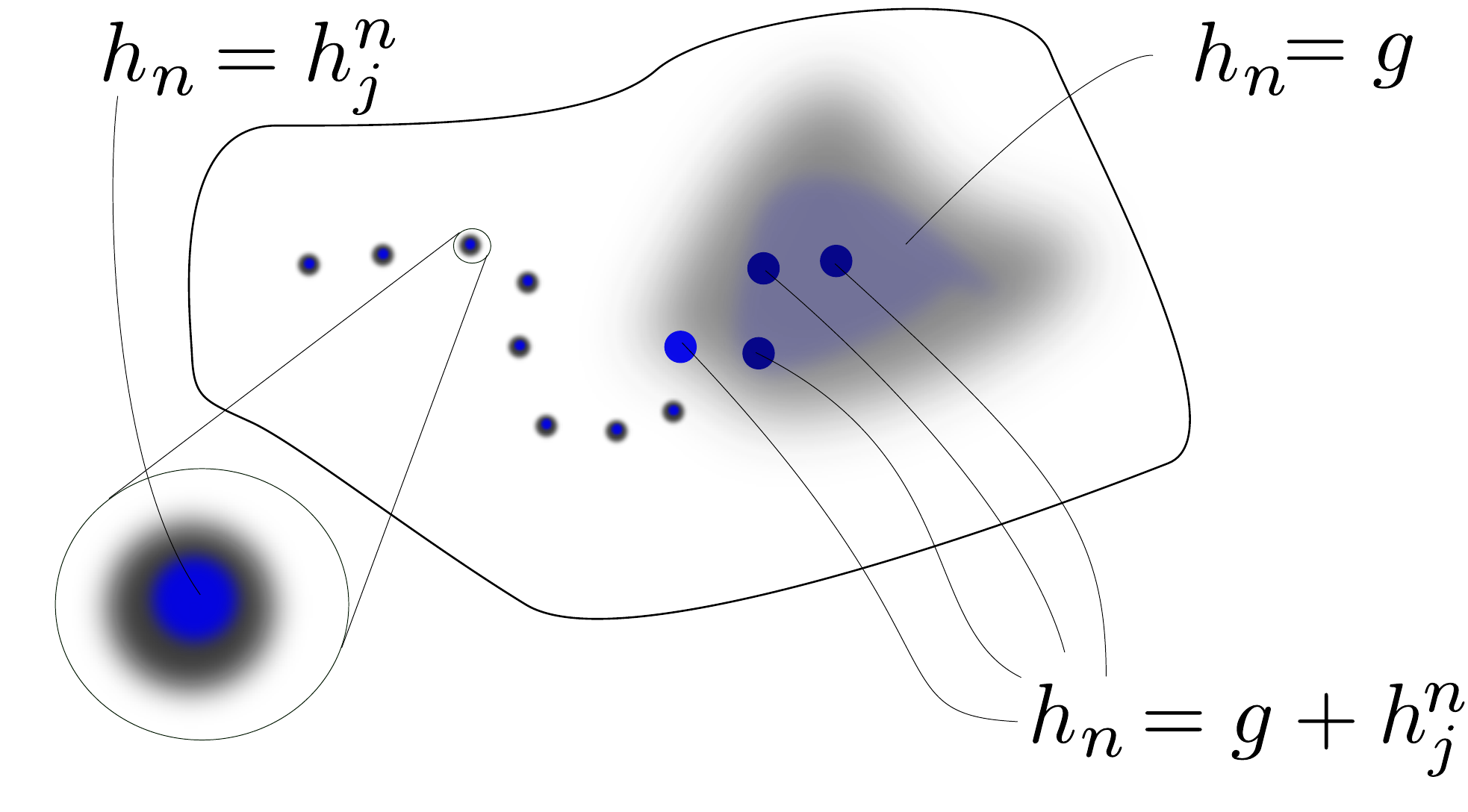}\label{approx3}}
\caption{Construction of the approximating sequence $(\varphi_n, h_n \F^{\varphi_n })$}
\end{figure}

\begin{proof}
Let $n\in\N$, and consider a grid of open cubes $\{Q_j^n\}_{j\in \N}$ of edge length $1/n$.
Without loss of generality, we can assume $\F^\phi(\partial Q_j^n\cap\Omega)=0$ for every $j\in\N$. Write $\mu=g\F^{\phi}+\mu^{\perp}$, where $g:=\frac{\d \mu}{\d \F^\phi}$. We divide the proof in four steps.\\

\emph{Step 1.} We claim that there exists a sequence $\{\lambda_n\}_{n\in \N}\subset \mathcal{M}^+ (\Omega)$ of the form
\begin{equation}\label{eq:lambdak}
\lambda_n:=\sum_{j=1}^{M_n} \mu^{\perp}(Q_j^n) \delta_{x^n_j}\,,
\end{equation}
with
 $x_j^n\in Q_j^n\cap \spt \mu^{\perp}$,  for all $j=1,\ldots,M_n$ such that
\begin{equation}\label{eq:approxdelta}
\lim_{n\to+\infty} [\, \d_{\mathcal{M}}(\mu^{\perp},\lambda_n) + \Theta_{\zeta} |\mu^{\perp}(\Omega)-\lambda_n(\Omega)| \,]= 0.
\end{equation}
Define
	\[
	\mathcal{J}^n_{\mu^{\perp}}:=\{ j\in \N \ \ | \ \  \spt\mu^{\perp}\cap Q_j^n\neq \emptyset\},
	\]
For every $j\in\mathcal{J}^n_{\mu^{\perp}}$, choose $x_j^n \in\spt\mu^{\perp}\cap  Q_j^n$ with
\[
\frac{\mathrm{d}\F^\phi}{\mathrm{d}\mu^\perp}(x_j^n)=0
\]
and define
	\[
	\lambda_n:=\sum_{j\in \mathcal{J}^n_{\mu^{\perp}}} \mu^{\perp}(Q_j^n) \delta_{x_j^n}.
	\]
Note that
\[
\sup_{n\in\N} \lambda_n(\Omega)<+\infty\,.
\]
Let $E\subset\subset\Omega$ be a bounded Borel set with $\mu^{\perp}(\partial E)=0$. Set
	\[
		 \mathcal{J}^n_{\mu^{\perp},\partial E}:=\{j\in\mathcal{J}^n_{\mu^{\perp}} \ | \
		 	Q_j^n\cap \partial E \neq \emptyset\}.
	\]
Fix $\eta>0$ and let $U\supset \partial E$ be an open set such that $\mu^{\perp}(U)<\eta$.
Take $n_0\in\N$ large enough so that $j\in \mathcal{J}^n_{\mu^{\perp},\partial E}$ implies $Q_j^n\cc U$ for all $n>n_0$.
Note that 
	\begin{align}\label{eq:nuk1}
	|\lambda_n(E)-\mu^{\perp}(E) | &\leq \sum_j | \mu^\perp(\overline{Q^n_j}) - \lambda_n( \overline{Q^n_j} ) |  \nonumber \\	
	&\leq \sum_{\mathcal{J}^n_{\mu^{\perp},\partial E}} \mu^{\perp}( \overline{Q_j^n} )
			+ \sum_{\mathcal{J}^n_{\mu^{\perp},\partial E}} \lambda_n( \overline{Q_j^n} )  \nonumber \\
	&\leq 2\mu^{\perp}(U) < 2\eta \,.
	\end{align}
Therefore
	\begin{equation}\label{eq:nuk2}
	\lim_{n\rightarrow +\infty} |\lambda_n(E)-\mu^{\perp}(E)|\leq 2\eta \,.
	\end{equation}
Since $\eta>0$ is arbitrary, from \eqref{eq:nuk1} and \eqref{eq:nuk2} we get
 	\[
	\lim_{n\rightarrow +\infty} |\lambda_n(E)-\mu^{\perp}(E)|=0.
	\]
This proves the claim.\\

\emph{Step 2.} Let $\{\lambda_n\}_{n\in \N}\subset \mathcal{M}^+(\Omega)$ be the sequence provided by Step 1. Then
	\[
	\lambda_n:=\sum_{j=1}^{M_n} \mu^{\perp}(Q_j^n) \delta_{x^n_j}
	\]
for $x^n_1,\ldots,x^n_{M_n}\in \spt \mu^{\perp}$.
Note that, for every $n\in\N$, the cubes $\{Q^n_j\}_{j=1}^{M_n}$ are pair-wise disjoint.
The idea is to locally deform the function $\phi$ around each point $x_j^n$ and define a corresponding density $h_j^n$ in a neighborhood of $x_j^n$.
Let
	\begin{align*}
	\mathcal{I}_a^n&:=\{j=1,\ldots, M_n\ \,|\, \ Q_j^n\cap \Omega\neq \emptyset, \  x_j^n\in \spt \F^{\phi}\}\,,\\
	\mathcal{I}_b^n&:=\{j=1,\ldots, M_n \ \,|\, \ Q_j^n\cap \Omega\neq \emptyset, \  x_j^n\notin \spt \F^{\phi}\}.
	\end{align*}
Fix $n\in\N$. We show how to recursively define $h^n_{j}$ and $\phi^n_{j}$. Set $\phi^n_0:=\phi$, and $h^n_0:=g$.
Assume $\phi^n_{j-1}$, and $h^n_{j-1}$ are given, and define $\phi^n_j$, and $h^n_j$ as follows.

\emph{Case one:} \textit{$j\in\mathcal{I}_a^n$}.
In this case, for any $r>0$ we have $\F^{\phi}(B_r(x^n_j))>0$.
Moreover using \cite[Theorem 2.22]{AFP}, we get
	\[
	\lim_{r\rightarrow 0} \frac{\mu^{\perp}(B_{r}(x_j^n))}{\F^{\phi}(B_{r}(x_j^n))}=+\infty.
	\]
In particular we can find $r_n<<1$ such that $B_{r_n}(x^n_j)\subset Q_j^n$ and
	\begin{equation}\label{eq:proprapprox1}
	\F^{\phi}(\partial B_{r_n}(x^n_j))=0,\quad \ \quad\quad \ 
	0<\F^{\phi}( B_{r_n}(x^n_j))<\frac{\mu^{\perp}( Q_j^n )}{M_n n} \\	 
	\end{equation}
for all $j\in \mathcal{I}_a^n$. Then we define $\phi^n_j:=\phi^n_{j-1}$, and
	\begin{equation*}
	h_j^n:=\frac{\mu^{\perp}(Q_j^n)}{\F^{\phi}(B_{r_n}(x^n_j))}\ca_{B_{r_n}(x^n_j)} 
		+ h^n_{j-1}\ca_{\Omega\setminus B_{r_n}(x^n_j)}\,.
	\end{equation*}

\emph{Case two:} \textit{$x_j\in\mathcal{I}_b^n$}.
Therefore, there exists $r_0>0$ such that for all $r_0>r>0$ we have
	\[
	\F^{\phi}(B_r(x_j^n))=0.
	\]
Fix $r_n<<1$ such that $B_{2r_n}(x_j^n)\subset Q_j^n$, and invoke property (Ad3) of Definition \ref{def:Admissible} with $A=B_{2r_n}(x_j^n)$, $U=B_{r_n}(x_j^n)$, $\e_n:= (M_n n)^{-1}  \min_{j=1,\ldots,M_n}\{\mu^{\perp}(Q_j^n)\} $ to find $ \phi_j^n\in X$ such that $\phi_j^n=\phi^n_{j-1}$ on $\Omega\setminus B_{r_n}(x_j)$, $\F^{\phi_j^n}(\partial B_{r_n}(x_j))=0$ and
	\begin{equation}\label{eq:proprapprox2}
\|\phi_j^n-\phi^n_{j-1}\|_{L^1}\leq \frac{  \mu^{\perp}(\Omega)}{M_n n}, \ \qquad \ 
0<\F^{\phi_j^n}(B_{r_n}(x^n_j) )<\frac{ \mu^{\perp}(Q_j^n)}{M_n n} 
	\end{equation}
for all $j\in \mathcal{I}_b^n$. Define
\begin{equation*}
	\begin{split}
	h_j^n&:=\frac{\mu^{\perp}(Q_j^n)}{\F^{\phi_j^n}(B_{r_n}(x_j))}\ca_{B_{r_n}(x^n_j)} 
		+ h^n_{j-1}\ca_{\Omega\setminus B_{r_n}(x^n_j)}\,.
	\end{split}
	\end{equation*}
Set $\varphi_n:=\phi^n_{M_n}$ and $h_n:=h^n_{M_n}$.
Note that $\varphi_n=\phi$  that $h_n=g$  outside $\bigcup_{j=1}^{M_n} B_{r_n}(x_j^n)$, that
\begin{equation}\label{eq:hklarge}
h_n\geq n M_n \quad\quad\quad \text{ on }\bigcup_{j=1}^{M_n} B_{r_n}(x_n)\,,
\end{equation}
and that, by construction, $\varphi_n\in X$.
Moreover $\F^{\varphi_n}(\partial B_{r_n}(x_j^n))=0$ and 
	\[
	\|\varphi_n-\phi\|_{L^1(\Omega)}\leq
		\sum_{j\in \mathcal{I}_b} \|\phi_j^n- \phi^n_{j-1}\|_{L^1}\leq \frac{\mu^{\perp}(\Omega)}{n}\,.
	\]
\emph{Step 3.} We claim that
\[
	\lim_{n\rightarrow+\infty} \int_{\Omega} \zeta(h_n)\d \F^{\varphi_n}
	 =\int_{\Omega} \zeta(g)\d\F^{\phi}+\Theta_{\zeta} \mu^{\perp}(\Omega).
	\]
Indeed, recalling Remark \ref{rem:locality}, we get
	\begin{equation}\label{Fixing1}
	\int_{\Omega} \zeta(h_n) \d\F^{\varphi_n}
		=\int_{\Omega\setminus \bigcup_{j=1}^{M_n} B_{r_n}(x_j^n)} \zeta(g)\d \F^{\phi}
			+\sum_{j=1}^{M_n} \frac{\zeta\left(h_n \right)}{h_n} \mu^{\perp}(Q_j^n)\,.
	\end{equation}
Fix $\delta_0>0$. Using \eqref{eq:hklarge} it is possible to take $n$ large enough so that
	\[
	\left|\frac{\zeta\left(h_n \right)}{h_n} -\Theta_\zeta\right|\leq \delta_0
	\]
for all $j=1,\ldots,M_n$. In particular
	\begin{align*}
	\left|\sum_{j=1}^{M_n} \frac{\zeta\left(h_n \right)}{h_n} \mu^{\perp}(Q_j^n)
		-\Theta_\zeta \lambda_n(\Omega)\right|&\leq\left|
			\sum_{j=1}^{M_n}  \left(\frac{\zeta\left(h_n \right)}{h_j^n} 
				- \Theta_\zeta\right)\mu^{\perp}(Q_j^n)\right|\\
	&\leq  \sum_{j=1}^{M_n}  \delta_0 \mu^{\perp}(Q_j^n)=\delta_0\lambda_n(\Omega).
	\end{align*}
The arbitrariness of $\delta_0>0$, together with \eqref{eq:approxdelta}, yield
	\begin{equation}\label{Fixing2}
	\lim_{n\rightarrow +\infty}\sum_{j=1}^{M_n} \frac{\zeta\left(h_n \right)}{h_n} \mu^{\perp}(Q_j^n)
		=\Theta_\zeta \mu^{\perp}(\Omega)\,.
	\end{equation}
Now, set
\[
A_n:=\bigcup_{j=1}^{M_n} B_{r_n}(x_j^n)\,,
\]
and note that from \eqref{eq:proprapprox1}, and \eqref{eq:proprapprox2} we get that
	\begin{equation}\label{eq:fphiAk}
		\F^{\phi}(A_n)\leq \frac{\mu^{\perp}(\Omega)}{n}\to0
	\end{equation}
as $n\to+\infty$.
Since $g\in L^1(\Omega;\F^{\phi})$, using \eqref{eq:flinearatinfinity} we get $\zeta(g)\in L^1(\Omega;\F^{\phi})$.
Therefore, the Lebesgue Dominated Convergence Theorem yields
	\begin{equation}\label{Fixing3}
		\lim_{n\to\infty}\int_{\Omega\setminus \bigcup_{j=1}^{M_n} B_{r_n}(x_j^n)} \zeta(g)\d \F^{\phi}
			=\int_\Omega \zeta(g)\d \F^{\phi}\,.
	\end{equation}
Using \eqref{Fixing1}, \eqref{Fixing2} and \eqref{Fixing3} we get the claim.\\
		
\emph{Step 4.} To complete the proof it remains to show that $h_n \F^{\varphi_n}\wt \mu$.
Note that
\[
\sup_{n\in\N} \int_\Omega h_n \d \F^{\varphi_n}<+\infty\,.
\]
Let $E\subset\subset \Omega $ be a bounded Borel set such that $\mu(\partial E)=0$.
Fix $\eta>0$, and take an open set $U\subset\Omega$ with $U\supset \partial E$ such that $\mu(U)<\eta$.
Set
	\[
	\mathcal{J}^n_{E}:=\{j\in \N  \ | \ Q_j^k \subset E\}, \qquad \ \mathcal{J}^n_{\partial E}
		:=\{j\in \N \ | \ Q_j^n \cap \partial E\neq \emptyset \}
	\]
and take $n_0\in\N$ large enough so that $j\in\mathcal{J}^n_{\partial E}$ implies $ Q_j^n\cc U$ for all $n>n_0$. 
Then
	\[
	h_n \F^{\varphi_n} (E)=\int_{E\setminus A_n } g\d \F^{\phi}
		+\sum_{j\in \mathcal{J}_{E}^n} \mu^{\perp}(Q_j^n)
	+ \sum_{j\in \mathcal{J}_{\partial E}^n}  \frac{\mu^{\perp}(Q_j^n)}{\F^{\varphi_n}( B_{r_n}(x_j^n))}\F^{\varphi_n}(E\cap B_{r_n}(x_j^n) ) \,.
	\]
Recalling \eqref{eq:lambdak}, we get
	\begin{align*}
	|h_n \F^{\varphi_n}(E)-(g\F^{\phi}+\lambda_n)(E)|&\leq
	\sum_{j=1}^{M_n} g  \F^{\phi} (B_{r_n}(x_j^n) ) +\sum_{j\in \mathcal{J}_{\partial E}^n}	
		\mu^{\perp}(Q_j^n)\left| 
			\frac{\F^{\varphi_n}(E\cap B_{r_n}(x_j^n) )}{\F^{\varphi_n}( B_{r_n}(x_j^n))} -1\right|\\
		&\leq\int_{A_n}g\d\F^{\phi}+\sum_{j\in \mathcal{J}_{\partial E}^n} \mu^{\perp}(Q_j^n) \leq \int_{A_n}g\d\F^{\phi}+\mu^{\perp}(U)\\
		&\leq  \int_{A_n}g\d\F^{\phi}+\eta.
	\end{align*}
Therefore, using \eqref{eq:fphiAk}, and the arbitrariness of $\eta>0$, we get
\begin{equation}\label{eq:limithk}
\lim_{n\to+\infty}|h_n \F^{\phi_n}(E)-(g\F^{\phi}+\lambda_n)(E)|=0\,.
\end{equation}
The claim follows at once by \eqref{eq:lambdak}, \eqref{eq:limithk} and the triangle inequality.
\end{proof}

Next result will allow us to consider only a special class of absolutely continuous couples. Let us introduce the following notation for partition of $\Omega$ into cubes.
Let $Q:=(-\frac{1}{2},\frac{1}{2})^d$. For $p\in\R^d$, and $\ell>0$, we define
	\[
	G_{p,\ell} :=\{(p+\ell z+\ell  Q)\cap \Omega \ | \ z\in \mathbb{Z}^d\}\, 
	\]
and we write $\partial G_{p,\ell} $ for
	\[
	\partial G_{p,\ell} := \Omega \cap  \left(\bigcup_{z\in \mathbb{Z}^d} (p+\ell z+\ell \partial Q)\right)\,.
	\]

\begin{definition}
We say that $(\phi,\mu)\in X\times \M^+(\Omega)$ is a \textit{regular absolutely continuous couple} if $\mu=g\F^{\phi}$ where $g\in L^1(\Omega,\F^{\phi})$ is of the form
\[
g=\sum_{i=1}^M \alpha_i \ca_{\Omega_i}\,,
\]
where $\alpha_1,\dots,\alpha_M\in (0,+\infty)$, and $\{\Omega_i\}_{i=1}^M=G_{p,\ell} $ for some $p\in\R^d$ and $\ell>0$, is such that
\[
\F^{\phi}(\partial G_{p,\ell} )=0\,.
\]
We denote by $ R(\Omega)$ the class of all regular absolutely continuous couples.
\end{definition}

\begin{proposition}\label{propo:fullRCVRYabsoluteCouple}
Let $\F\in \Ad$, $\zeta:[0,+\infty)\rightarrow (0,+\infty)$ be a convex function, and $(\phi,g \F^{\phi})\in X\times \mathcal{M}^+(\Omega)$, with $g\in L^1(\Omega,\F^{\phi})$.
Then there exists $\{g_n\}_{n\in\N}\subset L^1(\Omega,\F^{\phi})$ with $\{ (\phi, g_n\F^{\phi})\}_{n\in\N} \subset R(\Omega)$, such that
	\[
	g_n\F^{\phi}\wt g\F^{\phi}\,,
	\]
and
	\[
	\lim_{n\rightarrow +\infty} \int_{\Omega} \zeta(g_n) \d \F^{\phi}= \int_{\Omega} \zeta(g) \d \F^{\phi}.
	\]
\end{proposition}

\begin{proof}
For $n\in\N$ let $p_n\in\R^d$ be such that
\[
\F^\phi(\partial G_{p_n,1/n})=0\,
\]
Let $\{Q_j^n\}_{j=1}^{M_n}$ be the cubes such that $G_{p_n,1/n}=\{Q_j^n\cap\Omega\}_{j=1}^{M_n}$ and set
	\[
	\mathcal{J}^n_{\F^{\phi}}:=\{j=1,\ldots,M_n \ |  \ \F^{\phi}(Q_j^n\cap\Omega)>0\}\,.
	\]
For $j\in \mathcal{J}^n_{\F^{\phi}}$ set
	\begin{equation}
	\alpha_j^n:= \frac{1}{\F^{\phi}(Q_j^n\cap \Omega)}\int_{Q_j^n\cap \Omega} g \d \F^{\phi}\,,
	\end{equation}
and for $j\notin \mathcal{J}^n_{\F^{\phi}}$ set
\[
\alpha_j^n:=\frac{1}{nM_n}\,.
\]
Define
	\[
	g_n:=\sum_{j=1}^{M_n} \alpha_j^n \ca_{Q_j^n}\,.
	\]
We claim that $g_n\F^{\phi}\wt g\F^{\phi}$.
Note that
\[
\sup_{n\in\N} \int_\Omega g_n \d \F^\phi<+\infty\,.
\]
Let $E\subset\subset\Omega$ be a bounded Borel set with $g\F^{\phi}(\partial E)=0$.
Fix $\eta>0$, and let $U\supset \partial E$ be an open set with $g\F^{\phi}(U)\leq \eta$.
Let
\[
\mathcal{J}^n_{\partial E}:=\{\, j=1,\dots, M_n \,\,|\,\, Q_j^n \cap \partial E\neq\emptyset \,\}\,,
\]
and take $n\in\N$ large enough so that $j\in \mathcal{J}^n_{\partial E}$ implies $Q_j^n\subset U$ for all $n\geq n_0$.
Then
\begin{align*}
|g_n\F^\phi(E)-g\F^\phi(E)| &\leq \sum_{j\in \mathcal{J}^n_{\partial E}}
	\left|\,\int_{Q^n_j\cap E} g_n-g \d \F^\phi \,\right| + \frac{1}{n} \\
&\leq\sum_{j\in \mathcal{J}^n_{\partial E}} \left|\, \frac{\F^\phi(Q^n_j\cap E)}{\F^\phi(Q^n_j)} -1 \,\right|
	\int_{Q^n_j} g \d \F^\phi +\frac{1}{n}\\
&\leq g\F^\phi(U) +\frac{1}{n}\\
&\leq \eta+\frac{1}{n}\,.
\end{align*}
Therefore
\[
\lim_{n\to+\infty} |g_n\F^\phi(E)-g\F^\phi(E)|\leq \eta\,.
\]
Since $\eta>0$ is arbitrary, we get the claim.
To conclude the proof we have to show that
\[
\lim_{n\to+\infty} \int_\Omega \zeta(g_n) \d \F^{\phi}=\int_\Omega \zeta(g) \d \F^\phi\,.
\]
Using the convexity of $f$, we have
	\begin{align*}
	\int_{\Omega} \zeta (g_n)\d \F^{\phi}&=\sum_{j=1}^{M_n}	\int_{Q_j^n\cap \Omega} \zeta(\alpha_j^n)\d \F^{\phi}\\
	&=\sum_{j\in \mathcal{J}_{\F^{\phi}}^n}	\F^{\phi}(Q_j^n\cap \Omega) \zeta \left(\frac{1}{\F^{\phi}(Q_j^n\cap\Omega)}
		\int_{Q_j^n\cap\Omega} g\d\F^{\phi}\right)\\
	&\leq \sum_{j\in \mathcal{J}_{\F^{\phi}}^n}	\int_{Q_j^n\cap \Omega}  \zeta(g)\d\F^{\phi}\\
	&\leq \int_{\Omega}  \zeta(g)\d\F^{\phi}\,.
	\end{align*}
Therefore
\begin{equation}\label{eq:limsupfgk}
\limsup_{n\to+\infty} \int_\Omega \zeta(g_n) \d \F^{\phi} \leq \int_\Omega \zeta(g) \d \F^\phi\,.
\end{equation}
On the other hand, using $g_n\F^{\phi}\wt g\F^{\phi}$, the convexity of $f$, and \cite[Theorem 2.34]{AFP}, we get
\begin{equation}\label{eq:liminffgk}
\int_\Omega \zeta(g) \d \F^\phi \leq \liminf_{n\to+\infty} \int_\Omega \zeta(g_n) \d \F^{\phi}\,.
\end{equation}
Using \eqref{eq:limsupfgk}, and \eqref{eq:liminffgk} we conclude.
\end{proof}

\begin{proposition}\label{propo:fromConvToSubConv}
Let $\F\in\Cllc$, and $(\phi,g \F^{\phi})\in R(\Omega)$.
Then there exists $\{(\varphi_n, h_n \F^{\varphi_n})\}_{n\in\N}\subset R(\Omega)$ such that
	\[
	\lim_{n\to+\infty} d( (\varphi_n, h_n\F^{\varphi_n}),(\phi, g\F^{\phi}) )
	\]
and
	\[
	\lim_{n\to+\infty}  \int_\Omega \psi^{cs}(h_n)\d \F^{\varphi_n}
		=\int_\Omega\psi^c(g)\d \F^{\phi}\,.
	\]
Moreover, if we write $h_n=\sum_{i=1}^M \kappa_i^n \ca_{\Omega^n_i}$, we can ensure that $\kappa^n_i\leq t_0$ for all $i=1,\ldots, M_n$, and $n\in\N$, where $t_0>0$ is given by Lemma \ref{lem:finalConv}.
\end{proposition}

\begin{proof}
If $t_0=+\infty$, then there is nothing to prove, since this would mean that $\psi^{cs}=\psi^c$.
Thus, assume $t_0\in(0,+\infty)$.
Write $g=\sum_{i=1}^M \alpha_i \ca_{\Omega_i}$, where $\alpha_i>0$, and let
\[
\mathcal{I}:=\{ i=1,\ldots, M \ | \ \alpha_i > t_0\}\,.
\]
Define the function
	\[
	f:=\sum_{i=1}^M\beta_i\ca_{\Omega_i}\,,
	\]
where
\begin{equation}\label{eqn:defBETA}
	\beta_i:=\left\{
	\begin{array}{ll}
	\frac{\alpha_i}{t_0} & \quad \text{if $i\in \mathcal{I}$} \,,\\
	1 & \quad \text{if $i\notin \mathcal{I}$}\,,
		\end{array}
	\right.
	\end{equation}
Let $\{\phi_n\}_{n\in \N}\subset X$ be the sequence given by Definition \ref{def:Admissible} such that
$\phi_n\to\phi$ in $X$, and 
\begin{equation}\label{eq:convFphin}
\F^{\phi_n}\wt f\F^\phi\,.
\end{equation}
In particular, for all $i=1,\dots,M$, it holds $\F^{\phi_n}(\partial\Omega_i)\to0$.
It is then possible, for all $n$ sufficiently large, to choose $\delta_n>0$ such that
\begin{equation}\label{eq:approxboundary}
\F^{\phi_n}\left( (\partial\Omega_i)_{\delta_n} \right)<\frac{1}{n}\,,
\end{equation}
for all $i=1,\dots,M$.
Here $(\partial\Omega_i)_{\delta_n}:=\{ x\in\R^d \,\,|\,\, \mathrm{dist}(x,\partial\Omega_i)<\delta_n \}$.
By definition of regular absolutely continuous couple we have that
$\{\Omega_i\}_{i=1}^M = G_{p,\ell} $ for some $p\in \R^d$, $\ell>0$.
Since $\F^{\phi_n}$ is a Radon measure, it is possible to slightly translates the underline grid of cubes of a small vectors so that $G_{p+v,\ell}$ do not charge energy $\F^{\phi_n}$ on $\partial G_{p+v,\ell}(\Omega)$.
More precisely, we can find a sequence $\{v_n\}_{n\in\N}\subset\R^d$ with $|v_n|<\delta_n$, such that
	\[ 
	\{ \widetilde{\Omega}^n_i\}_{i=1}^M=G_{p+v_n,\ell}
	\]
and
\[
\F^{\phi_n}(\partial \widetilde{\Omega}^n_i)=0\,,
\]
for all $i=1,\dots,M$ and $n\in\N$. 
Define, for $n\in\N$,
	\[
	h_n:=\sum_{i\in \mathcal{I}}t_0\ca_{\widetilde{\Omega}^n_i}
		+\sum_{i\notin \mathcal{I}}\alpha_i\ca_{\widetilde{\Omega}^n_i}\,.
	\]
We claim that $h_n\F^{\phi_n}\wt g\F^{\phi}$.
Indeed,
\[
\sup_{n\in\N} \int_\Omega h_n \d \F^{\phi_n}<+\infty
\]
and, for any bounded Borel set $E\subset\subset  \Omega$ with $\F^{\phi}(\partial E)=0$, we have that
	\begin{align}\label{eq:firstineq}
	\left| \int_E h_n\d \F^{\phi_n}- \int_E g\d\F^{\phi}\right|&\leq 
		\sum_{i\in \mathcal{I} }  | t_0 \F^{\phi_n}(\widetilde{\Omega}^n_i\cap E)
			- \alpha_i \F^{\phi}(\Omega_i\cap E)| \nonumber \\
	&\hspace{0.6cm}+ \sum_{i\notin \mathcal{I} } \alpha_i \left| \F^{\phi_n}(\widetilde{\Omega}^n_i\cap E)
		-   \F^{\phi}(\Omega_i\cap E)\right| \nonumber \\
	&\leq \sum_{i\in \mathcal{I} }  | t_0 \F^{\phi_n}(\Omega_i\cap E)-\alpha_i \F^{\phi}(\Omega_i\cap E)|
		\nonumber \\
	&\hspace{0.6cm}+t_0 \sum_{i\in \mathcal{I} } |  \F^{\phi_n}(\widetilde{\Omega}^n_i\cap E)
			-\F^{\phi_j}(\Omega_i\cap E)| \nonumber \\
	&\hspace{0.6cm}+\sum_{i\notin \mathcal{I} } \alpha_i |  \F^{\phi_n}(\widetilde{\Omega}^n_i\cap E)
				-\F^{\phi_n}(\Omega_i\cap E)| \nonumber \\
	&\hspace{0.6cm}+\alpha_i \sum_{i\not\in \mathcal{I} } |  \F^{\phi_n}(\widetilde{\Omega}^n_i\cap E)
			-\F^{\phi_n}(\Omega_i\cap E)| \nonumber \\	
	&\leq \sum_{i\in \mathcal{I} }  | t_0 \F^{\phi_n}(\Omega_i\cap E)-\alpha_i \F^{\phi}(\Omega_i\cap E)| 
		+C\F^{\phi_n}\left( (\partial\Omega_i)_{\delta_n} \right) \nonumber \\
	&\leq \sum_{i\in \mathcal{I} }  | t_0 \F^{\phi_n}(\Omega_i\cap E)-\alpha_i \F^{\phi}(\Omega_i\cap E)|
		+C\frac{1}{n}\,,
	\end{align}
where $C:=\max\{t_0,\alpha_1,\dots,\alpha_M\}$, and in the last step we used \eqref{eq:approxboundary}.
We now observe that
$\F^{\phi}(\partial (\Omega_i \cap E))=0$.
Since $f\geq1$, this also implies $f\F^{\phi}(\partial (\Omega_i \cap E))=0$ for all $i\in\mathcal{I}$.
Using \eqref{eq:convFphin}, for all $i\in \mathcal{I}$ we get
	\[
	\lim_{n\rightarrow +\infty}\F^{\phi_n}(\Omega_i\cap E)= \int_{\Omega_i\cap E} f \d \F^{\phi}=\frac{\alpha_i}{t_0}\F^{\phi}(\Omega_i\cap E).
	\]
Therefore, from \eqref{eq:firstineq} we deduce that
\[
\lim_{n\to+\infty} \left| \int_{E} h_n\d \F^{\phi_n}- \int_{E} g\d\F^{\phi}\right| =0\,.
\]
To conclude the proof, \eqref{eq:convFphin} together with \eqref{eq:approxboundary} yield
\begin{align*}
\lim_{n\to+\infty}\int_{\Omega}	\psi^c (h_n)\d \F^{\phi_n}
&=\lim_{n\to+\infty}\left[\, \sum_{i\in \mathcal{I}} \F^{\phi_n}(\widetilde{\Omega}^n_i)
	\psi^c\left(t_0\right) + \sum_{i\notin \mathcal{I}}  \F^{\phi_n}(\widetilde{\Omega}^n_i)
		\psi^c\left(\alpha_i\right) \,\right]\\
	&=\sum_{i\in \mathcal{I}} \F^{\phi}(\Omega_i) \alpha_i\frac{\psi^{c}\left(t_0\right)}{t_0}
		+ \sum_{i\notin \mathcal{I}}  \F^{\phi}(\Omega_i) \psi^c\left(\alpha_i\right)\\
	&= \sum_{i\in \mathcal{I}} \F^{\phi}(\Omega_i) \psi^{cs}(\alpha_i)
		+ \sum_{i\notin \mathcal{I}}  \F^{\phi}(\Omega_i) \psi^{cs}\left(\alpha_i\right)\\
	&= \int_{\Omega} \psi^{cs}(g)\d \F^{\phi}\,,
\end{align*}
where in the third equality above we used Lemma \ref{lem:finalConv}.
\end{proof}

\begin{proposition}\label{propo:BorToConv}
Let $\F\in \Ad$, and $\{\F_n\}_{n\in\N}\in \mathrm{GA}(\F)$.
Then, for every $(\phi,g \F^{\phi})\in R(\Omega)$ there exists a sequence $\{(\varphi_n,h_n \F_n^{\varphi_n})\}_{n\in \N}\subset X\times \mathcal{M}^+(\Omega)$ such that
\begin{equation*}
	\lim_{n\to+\infty}\mathrm{d}( (\varphi_n,h_n \F_n^{\varphi_n}), (\phi, g \F^\phi) )=0\,,
	\end{equation*}
and
\begin{equation*}
\limsup_{n\rightarrow +\infty} 
	\int_{\Omega} \psi (h_n ) \d \F_n^{\varphi_n} \leq \int_{\Omega} \psi^c(g)\d \F^\phi\, .
	\end{equation*}
\end{proposition}

\begin{proof}
Write $g=\sum_{i=1}^M \alpha_i \ca_{\Omega_i}$, where $\alpha_i>0$.
By exploiting property (GA3) of Definition \ref{def:goodApproximating}, it is possible to find
$\{\varphi_n\}_{n\in \N}\subset X$ such that  $\F_n^{\varphi_n}$ is non-atomic for each $n\in \N$,
$\varphi_n\rightarrow \phi$ in $L^1(\Omega)$, $\F_n^{\varphi_n}\wt \F^{\phi}$, and $\F_n^{\varphi_n}(\Omega)\to\F^{\phi}(\Omega)$.
Using Lemma \ref{lem:convenvel}, for all $i=1,\dots,M$, and $n\in\N$, it is possible to find
$\lambda^i_n\in[0,1]$, and $s^i_n,t^i_n\in [0,+\infty)$ with
\begin{equation}\label{eq:alphai}
\lambda^i_n s^i_n+(1-\lambda^i_n)t^i_n=\alpha_i
\end{equation}
such that
	\begin{equation}\label{eq:psic}
	\lambda^i_n \psi(s^i_n)+(1-\lambda^i_n)\psi(t^i_n) \leq \psi^c\left(\alpha_i\right) + \delta_n ,
	\end{equation}
and
\begin{equation}\label{eq:approxFphi2}
\lim_{n\to+\infty} \Biggl(\, \sup_{\substack{i=1,\ldots,M} }
	\{s^i_n, t^i_n, \psi(s^i_n), \psi(t^i_n)\} \,\Biggr)
		\sum_{i=1}^M\F_n^{\varphi_n}(\partial \Omega_i) =0\,.
\end{equation}
where $\delta_n\to0$ as $n\to+\infty$.
Up to a subsequence, not relabeled, we can assume $\lambda^i_n\to\lambda^i\in[0,1]$, as $n\to+\infty$.

By invoking Lemma \ref{lem:Crumble}, for all $i=1,\ldots,M$, there exists a sequence of Borel sets $\{R^i_m\}_{m\in \N}$ with $R^i_m\subset \Omega_i$ having the following properties:
	\begin{itemize}
	\item[a)] $\F_n^{\varphi_n} \restr R^i_m\wt \lambda^i\F_n^{\varphi_n}\restr\Omega_i$ as $m\rightarrow +\infty$;
	\item[b)] $\F_n^{\varphi_n} \restr (\Omega_i\setminus R^i_m)\wt (1-\lambda^i)\F_n^{\varphi_n}\restr\Omega_i$ as $m\rightarrow +\infty$;
	\item[c)] $\F_n^{\varphi_n}(\partial R^i_m)=\F^\phi (\partial R^i_m)=0$ for all $n\in \N$, $m\in \N$, $i=1,\ldots,M$.
	\end{itemize}
Using the fact that $\F^{\phi_n}_n\wt\F^\phi$ and that $\F^\phi(\partial\Omega_i)=0$ for all $i=1,\dots,M$, it is possible to select a subsequence $\{m_n\}_{n\in\N}$ with $m_n\to+\infty$ as $n\to+\infty$ such that
\begin{equation}\label{eq:propertymn}
\F_n^{\varphi_n} \restr R^i_{m_n}\wt \lambda^i\F^\phi\restr\Omega_i\,,
\quad\quad\quad
\F_n^{\varphi_n} \restr (\Omega_i\setminus R^i_{m_n})\wt (1-\lambda^i)\F^\phi\restr\Omega_i\,,
\end{equation}
and
\begin{equation}\label{eq:psilambda}
\psi(s^i_n)|\,\F_n^{\varphi_n}(R^i_{m_n})-\lambda^i\F^\phi(\Omega_i) 
\,| + \psi(t^i_n) |\F_n^{\varphi_n}(\Omega_i\setminus R^i_{m_n}) -(1-\lambda^i)\F^\phi(\Omega_i)\,|  \to 0
\end{equation}
as $n\to+\infty$ for all $i=1,\dots,M$.
Define
	\[
	h_n:=\sum_{i=1}^{M} s^i_n\ca_{R^i_{m_n}}+ t^i_n \ca_{\Omega_i \setminus R^i_{m_n}}\,.
	\]
We claim that
\begin{equation}\label{gnamo}
	\lim_{n\rightarrow +\infty} \mathrm{d}_{\M} ( h_n \F_{n}^{\varphi_n}, g   \F^{\phi})=0.
	\end{equation}
Let $E\subset\subset \Omega$ be a bounded Borel set with $g\F^\phi(\partial E)=0$.
Then, using \eqref{eq:alphai}, and property c) above, we get
\begin{align*}
\left|\int_E h_n \d \F_n^{\varphi_n}-\int_E  g \d \F^\phi \right|
&\leq \sum_{i=1}^M s^i_n|\F_n^{\varphi_n}(R^i_{m_n}\cap E) - \lambda^i\F^\phi(E\cap \Omega_i)|\\
	&\hspace{1cm} +\sum_{i=1}^M  t^i_n | \F_n^{\varphi_n}(E\cap (\Omega_i\setminus R^i_{m_n}) )
		- (1-\lambda^i) \F^\phi(E\cap \Omega_i) |\\
	&\hspace{1cm}+\left(\, \sup_{i=1,\dots,M}\{s^i_n, t^i_n\}\,\right) 
		\sum_{i=1}^M\F_n^{\varphi_n}(\partial \Omega_i) \\
\end{align*}
Thanks to \eqref{eq:approxFphi2}, \eqref{eq:propertymn} and the the fact that $\lambda^i_n\to\lambda^i$, we conclude that
	\begin{equation}\label{eq:doublelimit}
	\lim_{n\rightarrow +\infty} \left|\int_{E} h_n \d \F_{n}^{\varphi_n} -\int_{E} g \d \F^{\phi}\right|=0.
	\end{equation}
Since
\[
\sup_{n\in\N} \,\int_\Omega h_n \d \F^{\phi_n}<+\infty\,,
\]
we obtain \eqref{gnamo}.
Finally, we note that
\begin{align*}
\int_{\Omega} \psi(h_n) \d \F_n^{\varphi_n}=
	&\sum_{i=1}^{M}\psi(s^i_n)\F_n^{\varphi_n}(R^i_{m_n})
		+\psi(t^i_n)\F_n^{\varphi_n}(\Omega_i\setminus R^i_{m_n})\\
		&\hspace{1cm}+ \sup_{i=1,\dots,M }\{\psi(s^i_j), \psi(t^i_j)\}
			\sum_{i=1}^M\F_n^{\varphi_n}(\partial \Omega_i)\\
&\leq \sum_{i=1}^ M [\, \lambda^i_n \psi(s^i_n) + (1-\lambda^i_n) \psi(t^i_n) \,]
		\F_n^{\varphi_n}(\Omega_i) \\
	&\hspace{1cm}+ \sup_{i=1,\dots,M } \{\psi(s^i_n),\psi(t^i_n)\}
		\sum_{i=1}^M\F_n^{\varphi_n}(\partial \Omega_i)\\
	&\hspace{1cm}+\sum_{i=1}^M \psi(s^i_n)|\,\F_n^{\varphi_n}(R^i_{m_n})-\lambda^i_n\F^\phi(\Omega_i) \,| \\
	&\hspace{1cm}+\sum_{i=1}^M \psi(s^i_n)|\,\F_n^{\varphi_n}(\Omega_i\setminus R^i_{m_n})
		-(1-\lambda^i_n)\F^\phi(\Omega_i) \,| \\
&\leq \sum_{i=1}^{M} \left(\psi^c(\alpha_i)+\delta_n \right) \F_n^{\varphi_n}(\Omega_i)
	+ \sup_{i=1,\dots,M }\{\psi(s^i_n), \psi(t^i_n)\} \sum_{i=1}^M\F_n^{\varphi_n}(\partial \Omega_i)\\
	&\hspace{1cm}+\sum_{i=1}^M \psi(s^i_n)\,|\,\F_n^{\varphi_n}(R^i_{m_n})-\lambda^i_n\F^\phi(\Omega_i) \,| \\
	&\hspace{1cm}+\sum_{i=1}^M \psi(t^i_n)\,|\,\F_n^{\varphi_n}(\Omega_i\setminus R^i_{m_n})
		-(1-\lambda^i_n)\F^\phi(\Omega_i) \,| \\
&= \int_{\Omega} \psi^c(g) \d \F^{\phi} +\delta_n\F^\phi(\Omega)
	+ \sup_{i=1,\dots,M }\{\psi(s^i_n), \psi(t^i_n)\} \sum_{i=1}^M\F_n^{\varphi_n}(\partial \Omega_i) \\
	&\hspace{1cm}+\sum_{i=1}^M \psi(s^i_n)\,|\,\F_n^{\varphi_n}(R^i_{m_n})-\lambda^i_n\F^\phi(\Omega_i) \,| \\
	&\hspace{1cm}+\sum_{i=1}^M \psi(t^i_n)\,|\,\F_n^{\varphi_n}(\Omega_i\setminus R^i_{m_n})
		-(1-\lambda^i_n)\F^\phi(\Omega_i) \,| \,.
\end{align*}
Thus, taking the limit as $n\to+\infty$ and using \eqref{eq:approxFphi2}, \eqref{eq:propertymn} together with the fact that $\lambda^i_n\to\lambda^i$, we get
\begin{equation}\label{gnamo2}
	\lim_{n\rightarrow +\infty} \int_{\Omega} \psi(h_n) \d \F_n^{\varphi_n} \leq \int_{\Omega} \psi^c(g) \d \F^{\phi}.  
	\end{equation}
This concludes the proof.
\end{proof}


\subsection{Proof of  Theorem \ref{thm:mainthm1}}

The liminf inequality follows from Proposition \ref{propo:lowerbound}.
Let $(\phi,\mu)\in X\times \mathcal{M}^+(\Omega)$, and write $\mu=g\F^{\phi}+\mu^{\perp}$.
Fix $\delta>0$.
By Proposition \ref{thm:density} there exist $\varphi_1 \in X$, and $h_1\in L^1(\Omega;\F^{\varphi_1})$ such that
	\begin{equation}\label{eq:thm1-approx1}
 	\d( (\varphi_1, h_1\F^{\varphi_1}),( \phi ,\mu))
 		+ \left|\int_{\Omega} \psi^{cs} (h_1) \d \F^{\varphi_1}
 			- \int_{\Omega} \psi^{cs} (g)\d \F^{ \phi} -\Theta^{cs} \mu^{\perp}(\Omega)\right|
 			\leq \frac{\delta}{4}\,.
	\end{equation}
Proposition \ref{propo:fullRCVRYabsoluteCouple} yields the existence of $h_2\in L^1(\Omega,\F^{\varphi_1})$ such that $(\varphi_1,h_2\F^{\varphi_1})\in R(\Omega)$ and
	\begin{equation}\label{eq:thm1-approx2}
 	\d( (\varphi_1, h_2\F^{\varphi_1}),(\varphi_1, h_1\F^{\varphi_1}))
 		+ \left|\int_{\Omega} \psi^{cs} (h_1) \d \F^{\varphi_1}
 			- \int_{\Omega} \psi^{cs} (h_2)\d \F^{ \varphi_1}\right|
 			\leq \frac{\delta}{4}\,.
	\end{equation}
Thanks to Proposition \ref{propo:fromConvToSubConv} we can find $\varphi_2\in X$, and $h_3\in L^1(\Omega,\F^{\varphi_2})$ such that $(\varphi_2,h_3\F^{\varphi_2})\in R(\Omega)$ and
	\begin{equation}\label{eq:thm1-approx3}
	\d( (\varphi_2, h_3\F^{\varphi_2}),(\varphi_1, h_2\F^{\varphi_1}) )
		+ \left|\int_{\Omega} \psi^c(h_3)\d \F^{\varphi_2}- 		
		\int_{\Omega} \psi^{cs}(h_2)\d \F^{\varphi_1}\right|\leq \frac{\delta}{4}\,.
	\end{equation}
Finally, let  $(\varphi_{n_{\delta}},h_{n_{\delta}}\F_{n_{\delta}}^{\varphi_{n_{\delta}}})\in X\times \mathcal{M}^+(\Omega)$ be given by Proposition \ref{propo:BorToConv} such that
	\begin{equation}\label{eq:thm1-approx4}
	\d( (\varphi_{n_{\delta}}, h_n\F_{n_{\delta}}^{\varphi_{n_{\delta}}}),(\varphi_2, h_3\F^{\varphi_2}) )\leq\frac{\delta}{4}\,,
	\end{equation}
and
	\begin{equation}\label{eq:thm1-approx5}
	\int_\Omega \psi(h_{n_{\delta}}) \d \F_{n_{\delta}}^{\varphi_{n_{\delta}}}\leq \int_\Omega \psi^{c}(h_3) d\F^{\varphi_3}
		+\frac{\delta}{4}
	\end{equation}
Therefore, from \eqref{eq:thm1-approx1}, \eqref{eq:thm1-approx2}, \eqref{eq:thm1-approx3}, \eqref{eq:thm1-approx4}, and \eqref{eq:thm1-approx5} we get
\[
\d( (\varphi_{n_{\delta}}, h_{n_{\delta}}\F_{n_{\delta}}^{\varphi_{n_{\delta}}}),( \phi ,\mu)) \leq \delta\,,
\]
and
\[
\int_\Omega \psi(h_{n_{\delta}}) \d \F_{n_{\delta}}^{\varphi_{n_{\delta}}} \leq \int_\Omega \psi^{cs}(g) d\F^\phi + \Theta^{cs} \mu^{\perp}(\Omega)+ \delta\,.
\]
We conclude by using Remark \ref{rem:Gamma}.
\qed


\section{Selected applications}\label{sec:appl}
In this section we assume the open set $\Omega\subset \R^d$ to have Lipschitz boundary.

\subsection{Perimeter functional} \label{sec:per}
As a first application of the general theory developed in the previous sections, we consider the perimeter functional.
In the following we will identify the space $X:=BV(\Omega;\{0,1\})$
with the space of sets with finite perimeter in $\Omega$.
We define the functional $\F:X\times \mathcal{A}(\Omega) \to[0,+\infty)$ as
\[
\F(\phi;A):=|D\phi|(A)=\mathcal{P}(\{\phi=1\};A)\,,
\]
where $\mathcal{P}(\{\phi=1\};A)$ denotes the perimeter of the set $\{\phi=1\}$ in $A$.
The following result has been proved in \cite{caroccia2017equilibria} (see \cite[Theorem 2]{caroccia2017equilibria}).

\begin{theorem}\label{thm:CCD}
Let $E\subset\R^d$ be a set of finite perimeter, and $f\in L^1(\partial^* E;[1,+\infty))$.
Then there exists a sequence of smooth bounded sets $\{E_n\}_{n\in\N}\subset\R^d$ with $\ca_{E_n}\to \ca_{E}$ in $L^1$, such that
\[
\lim_{n\to+\infty} \mathcal{P}(E_n;F)=\int_{\partial^* E\cap F} f \d \mathcal{H}^{d-1}\,,
\]
for all Borel sets $F\subset\subset \R^d$ with $\mathcal{P}(E;\partial F)=0$.
\end{theorem}

Using the above result, it is possible to obtain the following.

\begin{proposition}
The functional $\F$ is a purely lower semi-continuous admissible energy.
\end{proposition}

\begin{proof}
In order to show that the functional $\F$ is an admissible energy, we just need to prove property (Ad3), since the others are trivially satisfied.
Let $E\subset\Omega$ be a set of finite perimeter such that $|D\ca_E|(A)=0$ for some open set $A\subset\Omega$. Then $\ca_E$ is constant on $A$. Assume $\phi:=\ca_E=0$ on $A$. The other case can be treated similarly.
Let $U\subset\subset A$ be an open subset, and $\e>0$.
Pick $B_R(x)\subset\subset U$ and
for $r\in(0,R)$ set $\overline{\phi}:= \ca_E+\ca_{B_r(x)}$.
Then $\F(\overline{\phi};\partial U)=0$, $\|\overline{\phi}-\phi\|_{L^1}\leq \omega_d r^d$, and
	\[
	0 < \F(\overline{\phi};A) = \F(\overline{\phi};U) \leq d\omega_d r^{d-1}\,.
	\]
Taking
\[
r<\left(\frac{\varepsilon}{d\omega_d}\right)^{\frac{1}{d-1}}
\]
we get the desired result.

Finally, the fact that the energy $\F$ is purely lower semi-continuous follows by Theorem \ref{thm:CCD} localized in $\Omega$. Indeed, the wriggling procedure used in the proof of Theorem \ref{thm:CCD} is a local construction.
\end{proof}

\subsubsection{The Modica-Mortola approximation of the perimeter}
We now consider, for $\e>0$, the Modica-Mortola functional
$\F_\e: L^1(\Omega)\times\Oin\to[0,+\infty]$  defined as
\[
\F_\e(\phi;A):=\int_A \left[\, \frac{1}{\e}W(\phi) + \e|\nabla \phi|^2 \,\right] \d x\,,
\]
where $W\in C^0(\R)$ is non negative potential with at least linear growth at infinity, and such that $\{W=0\}=\{0,1\}$.
We report here the classical result by Modica (see \cite{modica1987gradient, modica1977limite}).

\begin{theorem}\label{thm:MM}
We have that $\F_\e\stackrel{\Gamma-L^1}{\rightarrow} \sigma_W\F$, where
\[
\sigma_W:=2\int_0^1 \sqrt{W(t)} \d t\,.
\]
Moreover, if $\{\phi_n\}_{n\in\N}\subset L^1(\Omega)$ is such that
\[
\sup_{n\in\N} \F_{\e_n}(\phi_n;\Omega)<+\infty\,,
\]
for some $\e_n\to0$, then, up to a subsequence (not relabeled), $\phi_n\to\phi$ in $L^1$, where $\phi\in BV(\Omega;\{0,1\})$
\end{theorem}

A careful analysis of the proof of the above results yields the following.

\begin{proposition}\label{prop:MMgood}
Let $\{\e_n\}_{n\in\N}$ be such that $\e_n\to0$, and set $\F_n:=\F_{\e_n}$ for $n\in\N$.
Then the sequence $\{\F_n\}_{n\in\N}$ is a good approximating sequence for $\F$.
\end{proposition}

\begin{proof}
We just have to prove property (GA3), being the others trivially satisfied.
The statement of Theorem \ref{thm:MM} holds for every open set $U\subset\Omega$ with Lipschitz boundary, and such that $|D\phi| (\partial U) =0$.
Therefore, by using Lemma \ref{lem:weakconvopen}, we get (GA3).
\end{proof}

Proposition \ref{prop:MMgood} allows us to use the abstract results proved in the previous section. In particular, we obtain the following.

\begin{proposition}\label{prop:resultperimeter}
Let $\psi:[0,+\infty)\to(0,+\infty)$ be a Borel function with $\inf \psi>0$.
For $\e>0$, consider the $\F_{\e}$-relative energy  
\[
\mathcal{G}_\e(\phi,\mu):=\int_\Omega\left[\, \frac{1}{\e}W(\phi) + \e|\nabla\phi|^2 \,\right] \psi(u) \d x\,,
\]
if $\phi\in H^1(\Omega)$, $\mu\in \mathcal{M}^+(\Omega)$ such that $\mu= u ( \frac{1}{\e}W(\phi) + \e|\nabla\phi|^2 )\L^d$, and $+\infty$ otherwise in $L^1(\Omega)\times \mathcal{M}^+(\Omega)$.
Then $\mathcal{G}_\e\stackrel{\Gamma}{\rightarrow}\mathcal{G}$ with respect to the $L^1 \times w^*$ topology, where
\[
\mathcal{G}(E,\mu):=\sigma_W\int_{\partial^* E} \psi^{cs}(u) \d \mathcal{H}^{d-1}
	+\Theta^{cs}\mu^{\perp}(\Omega)\,,
\]
is defined on any $E\subset\Omega$ set of finite perimeter and for the Radon-Nidodym decomposition of $\mu = u \left(\sigma_W \mathcal{H}^{d-1}\restr\partial^*E\right)+\mu^{\perp}$ with respect to $\mathcal{H}^{d-1}\restr\partial^*E$.
\end{proposition}

\begin{remark}
The functional $\mathcal{G}_\e$ has been used in \cite{ratz2006diffuse} as a phase field diffuse approximation for a model describing the evolution of interfaces in epitaxial growth with adatoms.
The authors worked with the special case $\psi(t):=1+\frac{t^2}{2}$.
In the same paper, it has been claimed that the solutions of the gradient flow of the phase field model converge to the solution of the sharp interface one.
This claim was supported by formal matching asymptotics.
It is worth to notice that the evolution equations for the sharp model do not account for the recession part and neither for the convex sub-additive envelope of $\psi$.

Proposition \ref{prop:resultperimeter} answer the question posed in the introduction.
\end{remark}


\subsection{Total variation functional}\label{sec:total}

In this section we generalize the result of Section \ref{sec:per} by considering the total variation functional defined over the whole class of functions of bounded variation.
Let
$\rho\in C^1(\Omega)\cap C^0(\overline{\Omega})$ such that 	
\[
0<\min_{\overline{\Omega}} \rho \leq\max_{\overline{\Omega}} \rho <+\infty\,,
\] 
and consider the energy $\F: BV(\Omega)\times \Oin \rightarrow [0,+\infty)$:
	\[
	\F(\phi;A):=\int_A \rho^2\,\d|D\phi|.
	\]

\begin{proposition}\label{prop:TVCLLC}
$\F\in \Cllc$.
\end{proposition}

We start by proving that the total variation is a purely lower semi-continuous functional.

\begin{proposition}[A wriggling result for Total Variation]\label{propo:wrigTV}
Let $\phi\in BV(\Omega)$ and $f\in L^1(\Omega,|D\phi|)$ with $f\geq 1$. Then there exists a sequence $\{\phi_k\}_{k\in \N} \subset BV(\Omega)$ such that
	\[
	\lim_{k\rightarrow +\infty} |D\phi_k|(E)=\int_E f \d |D\phi|\,,
	\]
for all Borel sets $E\subset\subset \Omega$ with $|D\phi|(\partial E)=0$.
\end{proposition}

The main technical step needed to get the above proposition is given by the following result.

\begin{lemma}\label{lem:techTV}
Fix $p\geq 1$, and let $Q_L\subset\R^d$ be a cube centered at the origin of edge length $L>0$.
Let $\phi\in C^{\infty}(Q_L)\cap C^0(\overline{Q}_L)$ and let $r>1$.
Then there exists a sequence of piece-wise $C^1$ maps $S_n:Q_L\rightarrow Q_L$ such that
\[
S_n =\Id \text{ on } \pa Q_L\,,
\quad\quad\quad
S_n\rightarrow \Id \text{ uniformly on }\overline{Q}_L\,,
\]
and
	\begin{itemize}
	\item[1)] $\displaystyle \lim_{n\rightarrow +\infty}
		\int_{Q_L} |\nabla  (\phi \circ S_n)|^p \d x=r \int_{Q_L} |\nabla  \phi|^p \d x$;
		\smallskip
	\item[2)] $\sup_{n\in \N} \|\nabla S_n\|_{L^\infty}<+\infty.$
	\end{itemize}
\end{lemma}

\begin{proof}
We divide the proof in several steps.\\
\text{}\\
\emph{Step one: $\phi$ affine.} Write $\phi(y)=y\cdot \nu +\phi_0$.
First of all we note that it suffices to show that, given $\beta>1$ and a cube $Q'\subset Q_L$ with two of its faces orthogonal to $\nu$, there exists a sequence of maps $T_n: Q'\rightarrow Q'$ such that
$T_n=\Id$ on $\partial Q'$, $T_n\rightarrow \Id$ uniformly on $\overline{Q'}$ and 
	\[
	\lim_{n\to+\infty}\int_{Q' } |\nabla  (\phi \circ T_n)|^p \d x 
		= \beta\int_{Q' } |\nabla \phi  |^p \d x  .
	\]
Indeed, by simply extending the map $T_n$ to the whole cube as the identity outside $Q'$ we get
	\begin{align}\label{eq:wrigglingbeta}
	 \int_{Q_L} |\nabla  (\phi \circ T_n)|^p\d x&=
	 	\int_{Q'} |\nabla  (\phi \circ T_n)|^p\d x
	 		+\int_{Q_L\setminus Q'} |\nabla \phi |^p  \nonumber \\
	& \rightarrow   \beta \int_{Q' } |\nabla   \phi  |^p \d x
		+ \int_{Q_L\setminus Q' } |\nabla   \phi  |^p \d x  \nonumber \\
	&=  |\nu|^p \left(\, \beta  |Q'|+ |Q\setminus Q'| \,\right) \nonumber \\
	&=\frac{1}{|Q_L|}\left(\, \beta  |Q'|
		+ |Q_L\setminus Q'| \,\right) \int_{Q_L}|\nu|^p \d x \nonumber \\
	&=\frac{1}{|Q_L|}\left(\, \beta  |Q'|+ |Q_L\setminus Q'| \,\right) \int_{Q_L} |\nabla \phi|^p \d x.
	\end{align}
Since the map
\[
\beta \mapsto \frac{1}{|Q_L|}\left(\, \beta  |Q'|+ |Q_L\setminus Q'| \,\right)
\]
is surjective on $[1,+\infty)$, given $r>1$ it is possible to find $\beta>1$ such that
\[
\frac{1}{|Q_L|}\left(\, \beta  |Q'|+ |Q_L\setminus Q'| \,\right)=r\,.
\]
Thus, using \eqref{eq:wrigglingbeta}, we conclude.

Therefore, we can assume without loss of generality, that $\nu=\rho \mathrm{e}_d$, $L=1, Q_L=Q$.
For $b>0$, let $s_b:\R\to\R$ be the periodic extension of the function
	\[
	t\mapsto (bt+b)\ca_{[-1,0]}(t)+(b-bt)\ca_{(0,1]}(t)\,,
	\]
for $t\in[-1,1]$.
For $n\in\N$ define the function $g_n:[-1,1]\rightarrow [0,+\infty)$ as
	\begin{equation}
		g_n(t):=\frac{1}{2n+1}s_b((2n+1)t).
	\end{equation}
Moreover, let $Q_n\subset Q$ be a cube of side length $1-\frac{1}{n}$, and fix a smooth cut-off function
$\varphi_n:\R\to[0,1]$ with $\varphi_n(t)=1$ for $|t|\leq 1-\frac{1}{n}$ and $\varphi_n(t)=0$ for
$|t|\geq 1$, and such that
	\begin{equation}\label{eq:cutoffvarphik}
		|\nabla \varphi_n|\leq Cn\,,
	\end{equation}
for some constant $C>0$ independent of $n$.
Define the function $T_n:Q\to Q$ as
	\[
	T_n(x):=x+\varphi_n(x\cdot\mathrm{e}_d) g_n(|\mathbb{P}(x)|\vee 1) \mathrm{e}_d\,,
	\]
where $\mathbb{P}(x):=x-(x\cdot \mathrm{e}_d)\mathrm{e}_d$.
We have
	\[
		\nabla T_n (x) = \Id
		+ \varphi_n (x\cdot \mathrm{e}_d)g'_n\left(|\mathbb{P}(x)|\vee 1\right)  \mathrm{e}_d
				\otimes \frac{\mathbb{P}(x)}{|\mathbb{P}(x)|}
		+g_n(|\mathbb{P}(x)|\vee 1) \varphi_n'(x\cdot \mathrm{e}_d) \mathrm{e}_d\otimes \mathrm{e}_d\,,
	\]
and
	\begin{align*}
		\nabla (\phi\circ T_n)(x)&=\rho \mathrm{e}_d \nabla T_n (x)\\
		&	=\rho \left[\, \mathrm{e}_d + \varphi_n(y\cdot \mathrm{e}_d) g'_n\left(|\mathbb{P}(x)|\vee 1\right)  
			\frac{\mathbb{P}(x)}{|\mathbb{P}(x)|}
		+g_n(|\mathbb{P}(x)|\vee 1) \varphi'(x\cdot \mathrm{e}_d) \mathrm{e}_d \,\right]\,.
	\end{align*}
Note that $T_n(Q)=Q$, $T_n=\Id$ on $\partial Q$ and that
\begin{equation}\label{eq:boundnabla}
\|\nabla T_n\|_{L^\infty}\leq C_{b}\,,
\end{equation}
where $C_b>0$ depends only on $b$.
Set $L_n:=1-\frac{1}{n}$, $Q_n:=(-L_n,L_n)^{d}$, $Q^{d-1}_n:=(-L_n,L_n)^{d-1}$.
Then
	\begin{align*}
	\int_{Q_n} |\nabla (\phi\circ T_n) (x) |^p\d x&=\rho^p \int_{Q_n} |\mathrm{e}_d \nabla T_n (x) |^p\d x \\
	&=\rho^p\int_{Q_n} (1+ g_n'^2(|\mathbb{P}(x)|\vee 1))^{\frac{p}{2}}\d x\\
	&=\rho^p\int_{-L_n/2}^{L_n/2} \int_{Q_n\cap \{ e_d\cdot x=t\}} 
		(1+ g_n'^2(|\mathbb{P}(x)|\vee 1))^{\frac{p}{2}}\d \H^{d-1}(x)\d t\\
	&=\rho^pL_n \int_{Q^{d-1}_n }  (1+ g_n'^2(|z|))^{\frac{p}{2}}\d \H^{d-1}(z)\\
	&=\rho^pL_n\left(\left(1-\frac{\omega_{d-1}}{2^{d-1}}\right)L_n^{d-1}
		+\left(\frac{L_n}{2}\right)^{d-1}\omega_{d-2}
			\int_{0}^{1} t^{d-2}(1+ b^2)^{\frac{p}{2}}\d t\right)\\
	&=\left( 1+\left((1+ b^2)^{\frac{p}{2}}-1\right)\frac{\omega_{d-1}}{2^{d-1}}\right)
		\int_{Q_n} |\nabla \phi (x)|^p\d x.
	\end{align*}
Hence
	\begin{align}\label{eqn:zio}
	\lim_{n\rightarrow +\infty} \int_{Q_n} |\nabla (\phi\circ T_n) |^p\d x
	=\left( 1+\left(\sqrt{1+b^2}-1\right)\frac{\omega_{d-1}}{2^{d-1}}\right)
	\int_Q |\nabla \phi (x)|^p\d x\,.
	\end{align}
On the other hand, using the definition of $g_n$, of $Q_n$, and \eqref{eq:cutoffvarphik}, we get
	\begin{align}\label{eq:zio1}
		\int_{Q\setminus Q_n} |\nabla (\phi \circ T_n)|^p\d x =0.
	\end{align}
Since the function
\[
b\mapsto \left(\, 1+\left((1+ b^2)^{\frac{p}{2}}-1\right)\frac{\omega_{d-1}}{2^{d-1}} \,\right)
\]
is surjective on $[1,+\infty)$, given $r>1$, it is possible to find $b>1$ such that, using \eqref{eqn:zio} and \eqref{eq:zio1}, we get
	\begin{equation}\label{eqn:zio2}
		\lim_{n\rightarrow +\infty} \int_Q |\nabla (\phi\circ T_n)|^p\d x
			=r\int_Q |\nabla \phi|^p\d x.
	\end{equation}
The required convergence on $T_n\rightarrow\Id$ follows at once.
\smallskip	

\emph{Step two:} \textit{$\phi\in C^{\infty}(Q_L)$}.
For $n\in\N$ consider the grid $\{Q_{i}^n\}_{i=1}^{n^d}$ of cubes of size $\frac{1}{n}$, with centers $x_{i}^n$ partitiong $Q$.
Fix $\delta>0$ and observe that there exists $n_0\in\N$ such that for all $n\geq n_0$ it holds
	\begin{equation}\label{eq:approxgradient}
	\max_{i=1,\dots n^d}\,\max_{y\in Q_i^n} |\nabla \phi(y)-\nabla \phi(x_i^n)|
		\leq \frac{\delta}{L^d}\,.
	\end{equation}
Set $\nu_i^n:=\nabla \phi(x_i^n)$, and define
\[
\psi_i^n(y):= y\cdot \nu_i^n+\phi(x_i^n)\,.
\]
For every $i=1,\ldots n^d$, thanks to the previous step, it is possible to find a map $T_i^n: \overline{Q_i^n}\to \overline{Q_i^n}$ with $T_i^n=\Id$ on $\partial Q_i^n$ such that
	\begin{equation}\label{eq:approxr}
	\left| \int_{ Q_i^n} |\nabla (\phi_i^n\circ T_i^n)|^p\d y
		- r\int_{Q_i^n} |\nabla \phi_i^n|^p\d y \right| \leq \frac{\delta}{n^d}.
	\end{equation}
Define $T_n:Q\to Q$ as
	\[
	T_n:=\sum_{i=1}^{n^d} T_i^n \ca_{Q_i^n}.
	\]
Using \eqref{eq:boundnabla} and \eqref{eq:approxgradient} we get
	\begin{equation}\label{eq:approxgradients1}
	\max_{y\in Q_L}\left|\nabla T_n(y) \nabla \phi (T_n(y))- \nabla T_n(y) \nu_i^n\right|
		\leq \frac{C \delta}{L^d}\,,
	\end{equation}
where the constant $C>0$ depends only on $L$ and $r$.
Therefore, from \eqref{eq:approxgradient}, \eqref{eq:approxr}, and \eqref{eq:approxgradients1} we get
\begin{align*}
\left|\, \int_{Q_L} |\nabla (\phi\circ T_n)|^p \d y - r\int_{Q_L} |\nabla \phi|^p \d y  \,\right|
	&\leq \sum_{i=1}^{n^d} \left|\, \int_{Q_i^n} |\nabla (\phi\circ T_n)|^p \d y
		- \int_{Q_i^n} |\nabla (\phi_i^n\circ T_n)|^p \d y  \,\right| \\
	&\hspace{0.5cm} + \sum_{i=1}^{n^d} \left|\, \int_{Q_i^n} |\nabla (\phi_i^n\circ T_n)|^p \d y
		- r \int_{Q_i^n} |\nabla \phi_i^n|^p \d y  \,\right| \\
	&\hspace{0.5cm} + r\sum_{i=1}^{n^d} \left|\, \int_{Q_i^n} |\nabla \phi_i^n|^p \d y
		- \int_{Q_i^n} |\nabla \phi|^p \d y  \,\right| \\
&\leq (1+C+r)\delta\,,
\end{align*}
where in the last step we used the inequality $|a^p-b^p|\leq |a-b|\,(|a|^{p-1}+|b|^{p-1})$.
Since $\delta>0$ is arbitrary, we conclude.
\end{proof}

It is now useful to introduce the following notation.
For a given Radon measure $\mu$ we define the \textit{class of regular multipliers $\mathfrak{m}(\Omega,\mu)$} as the family of $f\in L^1(\Omega,\mu)$ of the form
\[
f=\sum_{i=1}^{N} r_i \ca_{A_i}\,,
\]
where
\begin{itemize}
	\item[(i)] $r_i\geq 1$;
	\item[(ii)] $\{A_i\}_{i=1}^{N}$ is a finite family of pair-wise disjoint open subset of $\Omega$ with Lipschitz boundary;
	\item[(iii)]  $\mu(\partial A_i)=0$ for all $i=1,\ldots,N$.
\end{itemize}

By standard arguments of measure theory it is possible to prove the following density result.

\begin{lemma}\label{Lem:Fpiece-wiseconstant} 
Let $\mu\in\mathcal{M}^+(\Omega)$ be a Radon measure and pick $f\in L^1(\Omega,\mu)$ with $f\geq 1$ $\mu$-a.e. on $\Omega$. Then there exists a sequence $\{f_k\}_{k\in \N}\subset \mathfrak{m}(\Omega,\mu)$ such that $f_k\mu\wt f\mu$.
\end{lemma}

Therefore, we just need to provide the wriggling construction for functions $f\in \mathfrak{m}(\Omega,|D\phi|)$. This will be done in next result.

\begin{lemma}\label{lem:plscONregf}
Let $\phi\in BV(\Omega)$ and let $f\in \mathfrak{m}(\Omega,|D\phi|)$. Then there exists a sequence of maps $\{\phi_n\}_{n\in \N} \subset W^{1,1}(\Omega)$ such that $\phi_n\rightarrow \phi$ in $L^1$ and $|D\phi_n|\wt f |D\phi|$.
\end{lemma}
\begin{proof}
We divide the proof in three steps.\\

\textit{Step one:} Fix an open set $A\subset \Omega$, $\varphi\in C^{\infty}(\Omega)\cap C^0(\ov{\Omega})$ and $r\geq 1$.
The goal of this step is to construct a sequence $\{\varphi_k\}_{k\in \N}\subset W^{1,1}(\Omega)$ such that $\varphi_k=\varphi$ on $\Omega\setminus A$ and
	\[
	|D\varphi_k|\restr A \wt r|D\varphi|\restr A, \quad\quad\quad \varphi_k\rightarrow \varphi
		\quad \text{in $L^1$}.
	\] 
For each $k\in\N$, consider a grid of cubes $\{Q_j^k\}_{j\in \N}$ with side length $1/k$, and define the finite set of indexes
	\[
	\mathcal{J}_A^k:=\{ j \in \N \ | \ Q_j^k\cc A\}.
	\]
Using Lemma \ref{lem:techTV}, for each $Q_j^k$ there exists a smooth map $S_j^k:\Omega\to\Omega$ such that $S_j^k=\Id$ on $\Omega\setminus Q_j^k$ and
	\begin{equation}\label{e:kill}
	\left| |D (\varphi \circ S_j^k)|(Q_j^k)-r|D\phi|(Q_j^k)\right|+ \| \varphi \circ S_j^k- \varphi\|_{L^1(\Omega)}\leq \frac{1}{\# (\mathcal{J}_A^k) k}.
	\end{equation}
Define
	\[
	\varphi_k:=\sum_{j\in \mathcal{J}^k_A} (\varphi\circ S_j^k)\ca_{Q_j^k}\,.
	\]
Note that $\varphi_k\in W^{1,1}(\Omega)$, $\varphi_k=\varphi$ on $\Omega\setminus A$ and by construction $\varphi_k\rightarrow \varphi$ in $L^1$.
We now need to show that $|D\varphi_k|\restr A \wt r|D\varphi|\restr A$.
For, note that
\[
\sup_{k\in\N} |D\varphi_k|(A)<+\infty\,.
\]
Let $E \subset\subset  \Omega$ be a bounded Borel set such that $|D\varphi|(\partial E)=0$.
Take an open set $U\subset A$ with $U\supset \partial E$ such that $|D\varphi|(U)\leq \eta$ and define 
\[
	\mathcal{J}_{A,E}^k:=\{ j \in \mathcal{J}_A^k\ | \ Q_j^k\subset E \}\,,
	\quad\quad\quad
	\mathcal{J}_{A,\partial E}^k:=\{ j \in \mathcal{J}_{\partial E}^k\ | \ Q_j^k\cap \partial E\neq \emptyset \}.
	\]
Then
	\begin{align*}
	|D\varphi_k|(A\cap E)&=|D\varphi|((A\cap R_k)\cap E)+\sum_{j\in \mathcal{J}_{A,\partial E}^k} |D\varphi_k|(E\cap Q_j^k)+\sum_{j\in \mathcal{J}_{A, E}^k} |D\varphi_k|(Q_j^k)\\
	|D\varphi|(A\cap E)&=|D\varphi|((A\cap R_k)\cap E)+\sum_{j\in \mathcal{J}_{A,\partial E}^k} |D\varphi|(E\cap Q_j^k)+\sum_{j\in \mathcal{J}_{A, E}^k} |D\varphi|(Q_j^k),
	\end{align*}
	where 
\[
	R_k:=\Omega\setminus \left(\bigcup_{j\in \mathcal{J}_A^k}  Q_j^k\right).
\]
Take $k$ large enough so that $j\in \mathcal{J}_{\partial E}^k$ implies $Q_j^k\subset U$, and
	\[
	|D\varphi_k|(U)+|D\varphi|(U)\leq 3\eta\,.
	\]
Then, from the definition of $\varphi_k$ and \eqref{e:kill} we get
	\begin{align}\label{eq:estDphik}
| |D\varphi_k|(A\cap E) - r|D\varphi|(A \cap E)|\leq &\left|
	\sum_{j\in \mathcal{J}_{A, E}^k} |D\varphi_k|(Q_j^k) - r|D\varphi|(Q_j^k) \right|
		+|D\varphi|((A\cap R_k)\cap E)(r-1) \nonumber \\
&\hspace{1cm}+|D\varphi_k|(U)+|D\varphi|(U) \nonumber \\
&\leq 3\eta +\frac{1}{k}+ (r-1)|D\varphi|((A\cap R_k)\cap U).
	\end{align}
Since
\[
\bigcap_{k\in \N} R_k=(\Omega\setminus A)\,,
\]
by taking the limit as $k\rightarrow+\infty$ in \eqref{eq:estDphik} we get
\[
\lim_{k\to+\infty}| |D\varphi_k|(A\cap E) - r|D\varphi|(A \cap E)|\leq 3\eta\,.
\]
Since $\eta>0$ is arbitrary, we conclude that
	\[
	\lim_{k\rightarrow +\infty} |D\varphi_k|(A\cap E)=r|D\varphi|(A\cap E)\,,
	\]
getting $|D\varphi_k|\restr A \wt r|D\varphi|\restr A$.\\

\textit{Step two:} Let $\phi\in BV(\Omega)$, and $f\in L^1(\Omega,|D\phi|)$ of the form
\[
f=\sum_{j=1}^N r_i\ca_{A_i}\,,
\]
where $A_1,\dots,A_N\subset\Omega$ are pairwise disjoint open sets.
A density argument (See \cite[Proposition 3.21, Theorem 3.9]{AFP}) provides a sequence of maps
$\{\phi_{n}\}_{n\in\N}\subset C^{\infty}(  \Omega) \cap C^0(\overline{\Omega})$ such that
\[
\phi_{n}\rightarrow \phi \quad L^1( \Omega )\,,
\quad\quad\quad\quad
|D\phi_n|=|\nabla \phi_n|\L^n \wt |D\phi|\,.
\] 
For any $n\in \N$ and $i=1,\dots,N$, the previous step yields a sequence of maps $\{\phi_{i,n}^k\}_{k\in \N}\subset W^{1,1}(\Omega)$ such that $\phi_{i,n}^k=\phi_n$ on $\Omega\setminus A_i$ and
\[
|D\phi_{i,n}^k|\restr A_i \wt r_i |D\phi_n|\restr A_i\,,
\quad\quad\quad
\phi_{i,n}^k\rightarrow \phi_n\quad\text{ in } L^1(\Omega) \text{ as }k\rightarrow +\infty\,.
\]
Define
	\[
	\phi_n^k:=\sum_{i=1}^N \phi_{i,n}^k\ca_{A_i}\,.
	\]
Note that
\[
\phi_n^k\rightarrow \phi_n\quad\text{ in } L^1(\Omega)\,,
\quad\quad\quad
|D\phi_n^k|\wt f|D\phi_n|\,,
\]
as $k\to+\infty$.
Since by assumption $|D\phi|(\partial A_i)=0$ for all $i=1,\ldots, N$ we also have
\[
f|D\phi_n|\wt |D\phi|\,,
\]
as $n\rightarrow +\infty$.
Therefore, a diagonalization argument allows us to conclude.
\end{proof}

We are now in position to prove that the total variation is purely lower semi-continuous.

\begin{proof}[Proof of proposition \ref{propo:wrigTV}]
Fix $\phi\in BV(\Omega)$, and $f\in L^1(\Omega,|D\phi|)$.
By Lemma \ref{Lem:Fpiece-wiseconstant} applied to $\mu=|D\phi|$ we find a sequence of $\{f_k\}_{k\in \N} \in \mathfrak{m}(\Omega,|D\phi|)$ such that $f_k|D\phi|\wt f|D\phi|$.
Then, Lemma \ref{lem:plscONregf} applied on each $f_k\in \mathfrak{m}(\Omega,|D\phi|)$ yields a sequence of maps $\{\phi_n^k\}_{n\in \N}\subset W^{1,1}(\Omega)$ such that $|D\phi_n^k|\wt f_k|D\phi|$, $\phi_n^k\rightarrow \phi$ in $L^1(\Omega)$ as $n\rightarrow +\infty$.
Thus, we conclude by using a diagonalization argument.
\end{proof}

We can now conclude the proof of Proposition \ref{prop:TVCLLC}

\begin{proof}[Proof of Proposition \ref{prop:TVCLLC}]
Properties (Ad1) and (Ad2) follow easily by the definition of $\F$.

In order to prove property (Ad3) we argue as follows.
Let $\phi\in BV(\Omega)$, and $A\subset \Omega$ an open set such that $\F(\phi;A)=0$.
Fix $\e>0$, and an open set $U\cc A$.
Let $x\in U$ and $R>0$ such that $B_R(x)\subset U$.
Then $\phi\equiv c$ on $B_R(x)$.
Let $r\in(0,R)$ that will be chosen later, and set $\overline{\phi}:= \phi+\ca_{B_r(x)}$.
Then
\[
\F(\overline{\phi};\partial U)=0\,,\quad\quad\quad
\|\overline{\phi}-\phi\|_{L^1}\leq \omega_d r^d\,,
\]
and
	\[
	0<\F(\overline{\phi};A)\leq \int_{\partial B_r}\rho^2 \d\H^{d-1}\leq
		\left(\max_{\overline{\Omega}}\rho^2\right) d\omega_d r^{d-1}.
	\]
Taking
\[
r<\left(\frac{\varepsilon}{\left(\max_{\overline{\Omega}} \rho^2\right) d\omega_d}\right)^{\frac{1}{d-1}}
\]
we get the desired result.\\
\smallskip

Finally, we show that $\F$ is purely lower semi-continuous.
Fix $\phi\in BV(\Omega)$, $f\geq 1$, $f\in L^1(\Omega, \F^{\phi})$.
Let $\{\phi_n\}_{n\in\N}$ be the sequence given by Proposition \ref{propo:wrigTV} such that $\phi_n\to\phi$ in $L^1(\Omega)$ and
\begin{equation}\label{eq:propertyphin}
|D\phi_n|\wt f|D\phi|\,.
\end{equation}
We claim that
\[
\F^{\phi_n}\wt f \F^{\phi}\,.
\]

For each $k\in\N$, consider a partition of $\Omega$ into cubes $\{Q_j^k\}_{j=1}^{M_k}$ such that
$|D\phi|(\partial Q_j^k)=0$ and
	\begin{equation}\label{eq:rho1k}
	\left|\max_{Q_j^k\cap \overline{\Omega}} \rho^2 -\min_{Q_j^k\cap \overline{\Omega}} \rho^2 \right|
		\leq \frac{1}{k}
	\end{equation}
for all $j=1,\dots,M_k$.
Note that, for any open set $A\subset \subset \Omega$, the following estimates hold
	\[
	\left(\min_{\overline{Q_j^k}\cap \overline{A}} \rho^2 \right) |D\phi|(Q_j^k\cap  A)\leq	
		\F(\phi;Q_j^k\cap A )\leq 
		|D\phi|(Q_j^k\cap A)\left(\max_{\overline{Q_j^k}\cap \overline{A}} \rho^2 \right)\,.
	\]
Take a bounded Borel set $E\subset\subset\Omega$ such that $\F^{\phi}(\partial E)=0$.
Consider bounded open sets $E_1\subset\subset E\subset\subset E_2$.
Then, using \eqref{eq:rho1k}, we get
	\begin{align*}
	\F(\phi_n;E)&\geq \sum_{\substack{Q_j^k \cc E\\ Q_j^k\cap E_1\neq  \emptyset}}  \left(\min_{\overline{Q_j^k}} \rho^2 \right) |D\phi_n|(Q_j^k)\\
	&\geq \sum_{\substack{Q_j^k \cc E\\ Q_j^k\cap E_1\neq  \emptyset}}
		\left(\max_{\overline{Q_j^k}} \rho^2 -\frac{1}{k}\right) |D\phi_n|(Q_j^k)\,.
	\end{align*}
By taking the limit as $n \rightarrow+\infty$ on both sides, and using \eqref{eq:propertyphin} and \eqref{eq:rho1k}, we have
	\begin{align*}
	\lim_{n\rightarrow +\infty}\F(\phi_n;E)&\geq	
	\sum_{\substack{Q_j^k \cc E\\ Q_j^k\cap E_1\neq  \emptyset}}
		\int_{Q_j^k} \left(\max_{\overline{Q_j^k}} \rho^2 -\frac{1}{k}\right)f |D\phi|\\
	&\geq \sum_{\substack{Q_j^k \cc E\\ Q_j^k\cap E_1\neq  \emptyset}}
		\int_{Q_j^k} \left(\rho^2 - \frac{1}{k}\right)f \d |D\phi|\\
	&\geq \int_{E_1}  \rho^2 f \d |D\phi| -\frac{1}{k}\int_{\Omega} f\d|D\phi|\,.
	\end{align*}
Similar computation shows also
\begin{align*}
	\lim_{n\rightarrow +\infty}\F(\phi_n;E)&
		\leq\sum_{\substack{Q_j^k \cc E_2\\ Q_j^k\cap E\neq  \emptyset}}
			\int_{Q_j^k} \left(\min_{Q_j^k} \rho^2 +\frac{1}{k} \right)f \d|D\phi|\\
	&\leq \int_{E_2}  \rho^2 f \d|D\phi|+\frac{1}{k}\int_{\Omega} f\d|D\phi|.
	\end{align*}
Being the above valid for all $k\in \N$ yields
	\[
 \int_{E_1}  \rho^2 f \d |D\phi|\leq \lim_{n\rightarrow +\infty}\F(\phi_n;E)
 	\leq \int_{E_2}  \rho^2 f \d |D\phi|\,.
	\]
Moreover, since $E_1$ and $E_2$ are arbitrary and $\F^{\phi}(\partial A)=0$, we conclude that
	\[
	\lim_{n\rightarrow +\infty}\F(\phi_n;E)=\int_E  \rho^2 f \d |D\phi|\,.
	\]
Since 
\[
\sup_{n\in\N} \F^{\phi_n}(\Omega)<+\infty\,,
\]
we conclude that $\F^{\phi_n}\wt f \F^{\phi}$.
\end{proof}


\subsubsection{A non local approximation}
The nonlocal approximation of the weighted total variation we consider in this section is the one used by Garc\'{i}a-Trillos and Slep\v{c}ev in \cite{trillos2016continuum} in the context of total variation on graphs.

Let $\eta:[0,+\infty)\rightarrow [0,+\infty)$ be a compactly supported smooth function with $\int_0^{+\infty} \eta \d t =1$.
For $\varepsilon>0$ consider the energy $\F_{\e}:L^1(\Omega)\times \Ao(\Omega)\rightarrow [0,+\infty]$
defined as
	\[
	\F_{\e}(\phi;A):=\int_{A} \left(\int_{\Omega} \frac{|\phi(x)-\phi(y)|}{\e} \eta_{\e} (|x-y|)\rho(y)\d y\right)\rho(x)\d x
	\]
where $\eta_{\e}(t):=\eta(t \e^{-1})\e^{-d}$.
Fix a sequence $\{\varepsilon_n\}_{n\in\N}$ such that $\varepsilon_n\to0$ as $n\to+\infty$, and set $\F_n:=\F_{\varepsilon_n}$.
In \cite{trillos2016continuum} the authors proved the following $\Gamma$-convergence result.

\begin{theorem}\label{thm:DejanNicolas}
It holds that $\F_n\stackrel{\Gamma}{\rightarrow}\F$ with respect to the $L^1$ convergence, where
\[
\F(\phi;A):= \sigma_{\eta}\int_{A} \rho^2\d |D\phi|\,,
\]
and
\[
\sigma_{\eta}:=\int_0^{+\infty} t^n \eta(t)\d t\,.
\]
Moreover, if $\{\phi_n\}_{n\in\N}\subset L^1(\Omega)$ is a sequence for which
\[
\sup_{n\in\N} \F_n(\phi_n)<+\infty\,,
\]
then, up to a not relabeled subsequence, $\phi_n\to\phi$ in $L^1(\Omega)$, where $\phi\in BV(\Omega)$
\end{theorem}

It is possible to improve the above $\Gamma$-convergence result and obtain the following.

\begin{proposition}\label{propo:nonlocalTV}
The sequence $\{\F_n\}_{n\in\N}$ is a good approximating sequence for $\F$.
\end{proposition}

\begin{proof}
Properties (GA1) follows from the liminf inequality of Theorem \ref{thm:DejanNicolas}, while properties (GA2) is immediate from the definition of $\F_n$.

In order to prove property (GA3), we follow the same steps used in the proof of Theorem \ref{thm:DejanNicolas} (see \cite[Section 4.2]{trillos2016continuum}), that we briefly report here for the reader's convenience.
Given $\phi\in BV(\Omega)$, we extended it to a $BV$ function defined in the whole $\R^d$ in such a way that the extension, still denoted by $\phi$, is such that $|D\phi|(\partial\Omega)=0$. For any $\delta>0$, having set
	\[
	\Omega_{\delta}:=\{x\in \R^d \ | \ \dist(x,\Omega)\leq \delta\}
	\]
let $\{\phi_n\}_{n\in\N}\subset C^\infty(\Omega_{\delta})$ be a sequence such that $\phi_n\to\phi$ in $L^1(\Omega_{\delta})$, and
\[
 \rho^2|D\phi_n|\wt \rho^2 |D\phi| \quad \text{on $\mathcal{M}^+(\Omega_{\delta})$}.
 \]
Let $E\subset\subset \Omega$ be a Borel set with $|D\phi|(\partial E)=0$
Assume $[0,M]$ is the support of $\eta$. 
Then it holds that
	\begin{align*}
	\F_n(\phi_n;E)&:=\int_E \int_\Omega \frac{|\phi_n(x)-\phi_n(y)|}{\e_n}
		\eta_{\e_n}(|x-y|)\rho(x)\rho(y) \d y \d x\\
	&\leq \int_E \int_{B(x,M\e_n)} \eta_{\e_n}(|x-y|) \left|\, \int_0^1 \nabla\phi_n(x+t(y-x))\cdot(y-x) \d t  
			\,\right| \rho(y)\rho(x) \d y \d x \\
	&\leq \int_{E_{\e_n}} \int_{\{|h|\leq M\}} \eta(|h|) \int_0^1 |\nabla\phi_n(z)\cdot h|
		 \rho(z-t\e_n h)\rho(z+(1-t)\e_n h)	\d t \d h \d z \\
	&=\sigma_h \int_{E_{\e_n}} |\nabla \phi_n(z)|\rho^2(z) \d z + R_{n,k}\,, 	
	\end{align*}
where we used the change of variable $h:=\frac{x-y}{\e_n}$ and $z=x+t(y-x)$.
In the case $\rho$ is Lipschitz it is possible to prove that
\[
R_{n,k}\leq C \e_n \int_{E_{\e_n}} |\nabla \phi_n| \rho^2 \d z\,,
\]
for some constant $C>0$ independent of $k$.
Therefore
	\begin{equation*}
		\limsup_{n\rightarrow +\infty} \F_n(\phi;E)\leq \F(\phi;\overline{E})=\F(\phi;E)\,.
	\end{equation*}
\end{proof}

The result of Proposition \ref{propo:nonlocalTV} allows us to apply Theorem  \ref{prop:TVCLLC} to get the following.

\begin{proposition}\label{prop:resulttotal}
Let $\psi:[0,+\infty)\to(0,+\infty)$ be a Borel function with $\inf\psi>0$.
For $\e>0$ consider the functional
	\begin{equation*}
		\mathcal{G}_\e(\phi,\mu):=\int_{\Omega} \left[\, \int_{\Omega}\frac{|\phi(x)-\phi(y)|}{\e}
			\eta_{\e}(|x-y|)\rho(y)\d y \,\right] \psi(u(x))\rho(x) \d x\,,
	\end{equation*}
if $\phi\in L^1(\Omega)$, $\mu \in \mathcal{M}^+(\Omega)$ such that 
\[
\mu= u \rho \left(\int_{\Omega}\frac{|\phi(x)-\phi(y)|}{\e} \eta_{\e}(|x-y|)\rho(y)\d y\right) \L^d
\]
and $+\infty$ otherwise on $L^1(\Omega)\times \mathcal{M}^+(\Omega)$.
Then $\mathcal{G}_\e\stackrel{\Gamma}{\rightarrow}\G$ with respect to the $L^1\times w^*$ topology, where
	\[
		\mathcal{G}(\phi,\mu):=\sigma_{\eta}\int_{\Omega} \psi^{cs}(u)\rho^2\d |D\phi|+\Theta^{cs}\mu^{\perp}(\Omega)\,,
	\]
where we write $\mu=u\left(\sigma_{\eta}\rho^2|D\phi|\right)+\mu^{\perp}$.
\end{proposition}


\subsection{Realaxation of the p-Dirichlet energy}\label{sec:dir}

In this last section we would like to note that the choice of the $L^1$ convergence is not fundamental for the validity of the main results of this paper.
Indeed, Theorem \ref{thm:mainthm1} holds also if the $L^1$ converge is replaced by the $L^p$ convergence.
Of course, in Definition \ref{def:Admissible} and \ref{def:goodApproximating} one also has to replace the $L^1$ convergence with the $L^p$ one.

As an example of application of our results in the $L^p$ case, we consider, for $p>1$, the energy $\F:L^p(\Omega)\times\Oin\rightarrow [0,+\infty)$ given by
	\[
	\F(\phi;A):=\int_{A} |\nabla \phi|^p \d x \,.
	\]
when $\phi\in W^{1,p}(\Omega)$, and $+\infty$ otherwise.

\begin{proposition}
For any $p>1$, $\F\in \Cllc$.
\end{proposition}

\begin{proof}
It is easy to see that assumptions (Ad1),(Ad2) and (Ad3) are satisfied. Thus $\F\in \Ad$. 
The fact that $\F$ is purely lower semicontinuous is obtained by using a slight variation of the proof of Proposition \ref{propo:wrigTV} by applying Lemma \ref{lem:techTV} for $p>1$, Lemma \ref{Lem:Fpiece-wiseconstant} on $\mu=|\nabla \phi|^p\L^n$ and Lemma \ref{lem:plscONregf} suitably adapted.
\end{proof}

Noting that the constant sequence $\F_n:= \F$ is a good approximating sequence for $\F$ and using Corollary \ref{cor:relcllc} we obtain the following Proposition.

\begin{proposition}
The $L^p\times w^*$ lower semicontinuous envelope of $\mathcal{G}^{\F}(\phi,\mu)$ is
	\[
		\mathcal{G}(\phi,\mu):=\int_{\Omega} \psi^{cs}(u)|\nabla \phi|^p \d x +\Theta^{cs}\mu^{\perp}(\Omega)\,,
	\]
where we write $\mu=u|\nabla \phi|^p\L^n+\mu^{\perp}$.
\end{proposition}


\section{Appendix}

We here provide some technical results that has been used in the development of our arguments.

The first is a technical result, a \textit{Crumble Lemma}, namely a tool that allows us to disintegrate the domain of a non atomic measure $\F$ in sub-domains containing, asymptotically, a certain percentage of the total mass $\F$. This result plays a key role in the proof of Proposition \ref{propo:BorToConv}.

\begin{lemma}[Crumble Lemma]\label{lem:Crumble}
Let $\mu$ be a non-atomic positive Radon measure on $Q:=(-\frac{1}{2},\frac{1}{2})^d$, and $\{\mu_n\}_{n\in\N}$ be a sequence of Radon measure on $Q$.
 Then, for any $\lambda\in (0,1)$ there exists a sequence of Borel sets $\{R_j\}_{j\in \N}$, with $R_j\subset Q$ such that
	\begin{itemize}
	\item[a)] $\mu\restr R_j\wt \lambda\mu$;
	\item[b)] $\mu\restr Q\setminus R_j\wt (1-\lambda)\mu$;
	\item[c)] $\mu_n(\partial R_j)=0$ for all $n,j\in \N$.
	\end{itemize}
\end{lemma}

\begin{proof}
Fix $j\in\N$. It is possible to find $x_j\in\R^d$ such that $G_{x_j,1/j}$ is such that
$\mu(\partial G_{x_j,1/j})=0$. Let $\{Q_i^j\}_{j=1}^{M_j}$ denote the elements of the grid.
For every $i=1,\dots,M_j$, since $\mu$ is not atomic, it is possible to use \cite[Proposition 1.20]{FonLeo} in order to find a $\mu$-measurable set $S_i^j\subset Q^j_i$ such that $\mu(S^j_i)=\lambda\mu(Q^j_i)$.
We claim that there exists $\widetilde{S}_i^j\subset Q^j_i$ such that $\mu(\partial \widetilde{S}^j_i)=0$
\begin{equation}\label{eq:approxlambda}
|\mu(\widetilde{S}^j_i)-\lambda\mu(Q^j_i)|\leq\frac{1}{j M_j}\,.
\end{equation}
Indeed, consider the measure $\mu^j_i:=\mu\restr S^j_i$, and let $f_n:=\mu^j_i\ast\rho_n$, where $\{\rho_n\}_{k\in\N}$ is a sequence  of mollifiers. Let $n$ be large enough so that
\[
\left|\mu^j_i(Q)- \int_Q f_n \d x\right|<\frac{1}{2j M_j}\,.
\]
For $t>0$, let $\widetilde{S}^j_i:=\{ f_n\geq t\}$.
Using the fact that each $\mu_n$ is a Radon measure, and thus for all but countably many $t>0$ it holds $\mu_n(\partial \widetilde{S}^j_i)=0$, it is possible to find $t>0$ such that
$\mu_n(\partial \widetilde{S}^j_i)=0$ for all $n\in\N$, and
\[
\left|\int_{\widetilde{S}^j_i} f_n \d x - \int_Q f_n \d x\right|<\frac{1}{2j M_j}\,.
\]
Define, for each $j\in\N$
\[
R_j:=\bigcup_{i=1}^{M_j} \widetilde{S}^j_i\,.
\]
From the definition, it follows that $\mu(\partial R_j)=0$ for all $j\in\N$.
To prove that $\mu\restr R_j\wt \lambda\mu$, take a Borel set $E\subset Q$ with $\mu(\partial E)=0$.
Fix $\eta>0$, and let $U\supset\partial E$ be an open set with $\mu(U)<\eta$.
Let $j$ large enough so that $Q^j_i\cap\partial E\neq\emptyset$ implies $Q^j_i\subset U$.
Set $I:=\{ i=1,\dots M_j | Q^j_i\cap\partial E\neq\emptyset\}$.
Then
\begin{align*}
\mu(E\cap R_j) &= \sum_{i=1}^{M_j} \mu(\widetilde{S}^j_i\cap E) \\
&=\sum_{i\in I}\mu(\widetilde{S}^j_i\cap E) + \sum_{i\notin I} \mu(\widetilde{S}^j_i\cap E) \\
&\leq \eta + \sum_{i\notin I} \mu(\widetilde{S}^j_i\cap E) \\
&\leq \eta + \sum_{i\notin I} \lambda\mu(Q^j_i) +\frac{1}{j} \\
&\leq 2\eta + \lambda\mu(E)+\frac{1}{j} \,,
\end{align*}
where we used \eqref{eq:approxlambda}.
Therefore, we conclude by sending $j\to+\infty$, and using the arbitrariness of $\eta$.
\end{proof}

The second result is a test for the weak* convergence of measures.

\begin{lemma}\label{lem:weakconvopen}
Let $\{\mu_n\}_{n\in\N}$ be a sequence of nonnegative Radon measures on an open bounded set $\Omega\subset\R^d$ with Lipschitz boundary such that $\mu_n(U)\to\mu(U)$ for all open sets $U\subset\Omega$ with Lipschitz boundary such that $\mu(\partial U)=0$. Then $\mu_n\wt\mu$.
\end{lemma}

\begin{proof}
Let $E\subset\Omega$ be a Borel set with $\mu(\partial E)=0$.
We want to show that $\mu_n(E)\to\mu(E)$.
Fix $\e>0$.
Using the outer regularity of $\mu$, it is possible to find an open set $V\supset E$ such that
$\mu(V)< \mu(E)+\frac{\e}{2}$, and an open set $W\supset\partial E$ with $\mu(W)<\frac{\e}{2}$.
Since $\partial E$ is compact, there exists $r>0$ such that $B_r(x)\subset W$ for all $x\in\partial E$.
Let $U:=V\cup W$. Using mollification, it is possible to find an open set $U_2$ with Lipschitz boundary, such that $E\subset U_2\subset U$.
Then
\begin{equation}\label{eq:U1}
\mu(U_2)< \mu(E)+\e\,.
\end{equation}

Using the inner regularity of $\mu$, it is possible to find a compact set $K\subset E\setminus\partial E$ such that $\mu(K)>\mu(E)-\e$. Note that here we used the fact that $\mu(\partial E)=0$.
In the case $E\setminus\partial E=\emptyset$, we can just take $K=\emptyset$.
Since $K\subset E\setminus \partial E$, and $\Omega$ is bounded, we have that $\mathrm{dist}(K,\partial E)>0$.
Therefore, using mollifications, it is then possible to find an open set $U_1$ with Lipschitz boundary, such that $K\subset U_1\subset E\setminus\partial E$.
Then
\begin{equation}\label{eq:U2}
\mu(E)-\e<\mu(U_1)\,.
\end{equation}
Note that
\begin{equation}\label{eq:chainofinequalities}
\mu_n(U_1)\leq\mu_n(E)\leq\mu_n(U_2)\,.
\end{equation}
Since by assumption we have that
\[
\mu_n(U_1)\to\mu(U_1)\,,\quad\quad\quad
\mu_n(U_2)\to\mu(U_2)\,,
\]
from \eqref{eq:U1}, \eqref{eq:U2}, and \eqref{eq:chainofinequalities} we get
\[
\mu(E)-\e\leq\liminf_{n\to+\infty} \mu_n(E) \leq\limsup_{n\to+\infty}\mu_n(E)\leq\mu(E)+\e\,.
\]
Using the arbitrariness of $\e>0$, we conclude that $\mu_n(E)\to\mu(E)$, getting the desired result.
\end{proof}


\subsection*{Acknowledgements}
The authors are grateful to Irene Fonseca and Giovanni Leoni for having introduced them to the study of the problem.
The work of R.C. has been supported by the National Science Foundation under Grant No. DMS-1411646 during the period at Carnegie Mellon University, and by the grant EP/R013527/2 ``Designer Microstructure via Optimal Transport Theory" of David Bourne while at Heriot-Watt University.
The work of M.C has been supported by the Fundação para a Ciência e a Tecnologia (Portuguese Foundation for Science and Technology) through the Carnegie Mellon\textbackslash Portugal Program under Grant 18316.1.5004440 while in Lisbon, and by the grant ``Calcolo delle variazioni, Equazione alle derivate parziali, Teoria geometrica della misura, Trasporto ottimo" co-founded by Scuola Normale Superiore and the university of Florence during the last part of the work.


\end{document}